\documentclass[12pt,letterpaper]{article}
\usepackage[utf8]{inputenc}
\usepackage{amsmath}
\usepackage{amsfonts}
\usepackage{amssymb}
\usepackage{amsthm}
\usepackage{graphicx}
\usepackage[left=2cm,right=2cm,top=2cm,bottom=2cm]{geometry}
\usepackage{enumerate}
\usepackage{appendix}
\usepackage{mathtools}

\newtheorem{theorem}{Theorem}[section]
\newtheorem{lemma}[theorem]{Lemma}
\newtheorem{prop}[theorem]{Proposition}
\newtheorem{corollary}[theorem]{Corollary}
\theoremstyle{definition}
\newtheorem{definition}[theorem]{Definition}

\theoremstyle{remark}
\newtheorem{remark}[theorem]{Remark}

\DeclareMathOperator{\dv}{div}

\DeclareMathOperator{\supp}{supp}

\usepackage[backend=bibtex,style=alphabetic,sorting=nyt]{biblatex}

\numberwithin{equation}{section}

\usepackage{authblk}

\title{Well-Posedness for the Reaction-Diffusion Equation with Temperature in a critical Besov Space}
\author[1]{Chun Liu}
\author[1]{Jan-Eric Sulzbach}
\affil[1]{Department of Applied Mathematics, Illinois Institute of Technology, Chicago, IL 60616, United States}
\date{2/1/2021}
\begin{document}
\maketitle
\begin{abstract}
We derive a model for the non-isothermal reaction-diffusion equation.
Combining ideas from non-equilibrium thermodynamics with the energetic variational approach we obtain a general system modeling the evolution of a non-isothermal chemical reaction with general mass kinetics.
From this we recover a linearized model for a system close to equilibrium and we analyze the global-in-time well-posedness of the system for small initial data for a critical Besov space. 
\end{abstract}
\section{Introduction}
\subsection{Overview}
Reaction-Diffusion systems are a crucial part in science; from chemical reactions and predator-prey models to the spread of diseases.
These are just a few examples of the applications of reaction-diffusion systems.
Many of these systems have been studied over the last decades at a constant temperature or the equivalent in the respective field.
In recent years however, the focus shifted towards the analysis of non-isothermal models, that is systems with a non-constant temperature, leading to an additional non-linear equation to govern the temperature evolution.

For the chemical reaction-diffusion equation, the addition of a heat term not only adds an equation to the system, but also the material properties are affected, e.g with different local temperatures the viscosity and the heat conductivity can change.
For the chemical reaction in particular, the heat term also changes the reaction rate in the reaction.
These thermal effects in the chemical reaction equation have been studied from a chemical and engineering point in \cite{Anderson}, \cite{Chu} and \cite{Ross2008} or more recently \cite{Zarate} and \cite{Demirel06}.

In the mathematical theory of non-isothermal fluid mechanics there are two different ways to find and prove the existence of solutions. 
One method is to study the existence of weak solutions.
We refer to \cite{Novotny} for dealing with a general Navier-Stokes-Fourier system and to \cite{Bulicek} and \cite{Rocca} for some applications of the general theory.
Whereas the other method is to study the well-posedness of global solutions of the system \cite{Danchin}, \cite{Danchin2014} and \cite{DeAnna1}.
In this paper we follow the later approach studying the well-posedness of the reaction-diffusion system with temperature in a critical function space.
In addition, we present a new approach in the derivation of non-isothermal models in fluid mechanics. 
This approach follows the theory of classical irreversible thermodynamics \cite{Groot} and \cite{Lebon} and adds a variational structure to it, see \cite{Sulzbach} and \cite{Lai} for the application of this approach to the ideal gas system.
Other works that follow this idea but in a different setting or with a different variational structure are detailed in the book by \cite{Fremond} and the articles by \cite{GayBalmaz1}, \cite{GayBalmaz2} for example.\\

We consider the following system for the chemical reaction-diffusion system close to equilibrium, where we denote the concentration for each chemical species by $c_i$ for $i=A,B,C$ and the absolute temperature by $\theta$. 
Further, we denote the equilibrium state by $(\tilde{c}_A, \tilde{c}_B, \tilde{c}_C,\tilde{\theta})$ and the system then reads
\begin{align}
&\partial_t c_i -k^c\Delta c_i=-\sigma_i R_t+k^c\nabla\cdot\big(c_i\nabla \ln\theta\big),~~\text{ for } i=A,B,C\\ \begin{split}
&\sum_i k^\theta c_i\bigg[\partial_t\theta -k^\theta\bigg(\frac{\nabla c_i\cdot\nabla \theta}{c_i}+\frac{|\nabla \theta|^2}{\theta}\bigg)\bigg]=\kappa\Delta\theta +\sum_i\sigma_i k^\theta\theta R_t\\
 &~~~~~~+ (k^c)^2\sum_i \bigg[(\eta_i-1)\frac{|\nabla(c_i\theta)|^2}{c_i\theta}+\Delta(c_i\theta)\bigg]\end{split}
\end{align}  
where
\begin{align*}
R_t= k^c\ln\bigg(\frac{c_Ac_B}{c_C}\bigg)-k^\theta\ln\theta+k^c
\end{align*} 

The goal of this paper is to show the well-posedness of the above system in a critical Besov space.
By a critical space we mean a function space that has the same invariance with respect to scaling in time and space as the system itself. 
The scaling we consider is $(c_i,\theta)\to (c_i^\lambda,\theta^\lambda)$ where
\begin{align*}
c_i^\lambda(t,x)= c_i(\lambda^2 t,\lambda x)~~\text{and } \theta^\lambda(t,x)=\theta(\lambda^2 t,\lambda x).
\end{align*}
A natural function space to consider would be the Sobolev homogeneous space $\dot{H}^{d/2}$ but for the initial data in this space we cannot state the well-posedness result due to the lack of an algebraic structure. 
This can be overcome by considering the initial data of the problem in the critical Besov space $\dot{B}_{2,1}^{d/2}$.\\

This paper is structured as follows.
In the next section, we present an overview of non-equilibrium thermodynamics and the framework of our result.
This is followed by the derivation of the general model of a chemical reaction-diffusion system using these new ideas in Chapter 2.
In Chapter 3 we state the main definitions and theorems of the theory of Besov spaces that are used to show the well-posedness of the system.
The main result, i.e. the well-posedness of the non-isothermal chemical reaction-diffusion sytem close to equilibrium, and its proof can be found in Chapter 4.

\subsection{Non-Equilibrium Thermodynamics}
The theory of non-equilibrium thermodynamics derived from irreversible processes has been developed almost 100 years ago. 
Starting from the 1930s seminal work by Onsager (\cite{Onsager-a}, \cite{Onsager-b}) formulating his principles of irreversible thermodynamics with some underlying assumptions.
The idea is to extend the the concept of state from continuum thermostatics to a local description of material point in the continuum, i.e. every material point that constructs the continuum is assumed to be close to a local thermodynamic equilibrium state at any given instant.
Therefore, we can define the state variables and state functions such as temperature and entropy past their definition in equilibrium thermostatics.
This theory is known as Classical Irreversible Thermodynamics (CIT) (\cite{Groot}).
Besides the classical set of state variables, thermodynamic fluxes are introduced to describe irreversible processes.
In particular, the rate of change of entropy within a region is contributed by an entropy flux through the boundary of that region and entropy production inside the region.
In CIT the entropy flux only depends on the heat flux.
The non-negativity of the entropy production rate grants the irreversibility of the dissipative process and states the second law of thermodynamics.
The introduction of the local equilibrium hypothesis led to an impressive production of scientific research, but it is also the breaking point of the theory.
For systems far from equilibrium the CIT does no longer hold.

To extend the scope of the applications of non-equilibrium thermodynamics beyond the CIT, Truesdall, Coleman and Noll among others introduced Rational Thermodynamics (RT) (\cite{Truesdell}, \cite{Coleman}, \cite{Jou}). 
The main assumption of RT is that materials have memory, i.e. at any given time, dependent variables cannot be determined by only instantaneous values of independent variables, but by their entire history.
Thus speaking, the concept of state as known in CIT is modified and extended.
One drawback of RT is that temperature and entropy remain undefined objects.

In both CIT and RT, limitations of the possible form of the state and constitutive equations are obtained as a consequence of the application of the second law.
No restrictions however, are placed on the reversible parts, since they do not contribute to the entropy production.
By using a Hamiltonian structure restrictions on the reversible dynamics are provided.
An early version of a Hamiltonian framework for non-equilibrium thermodynamics was proposed by Grmela (\cite{Grmela}), based on a single generator.
This approach however was superseded by the work of Grmela and \"Ottinger (\cite{Grmela-a}, \cite{Grmela-b}) proposing the so called GENERIC formalism (General Equation for the Non-Equilibrium Reversible-Irreversible Coupling) and further developed by \"Ottinger (\cite{Oettinger}).
The GENERIC formalism relays on the generators, $E$ the total energy and $S$ the entropy.
This gives the theory more flexibility and emphasizes the central role played by the thermodynamic potentials.
The main achievement of GENERIC is its compact, abstract and general framework.
In this level of abstraction lies also the main difficulty of the formalism, its application to specific problems.

\subsection{Framework of this work}

The approach to non-equilibrium thermodynamics in this paper follows some of the ideas of the classical irreversible thermodynamics (CIT) and extend it to a variational framework.
The main assumption in this framework is that outside of equilibrium, there exists an extensive quantity $S$ called entropy which is a sole function of the state variables.\\

The structure of the derivation of the thermodynamic model is the following.
We introduce the free energy $\Psi$ as a basic quantity to define the material/ fluid properties.
From here, we derive the thermodynamic state function of the system.
In the second step, we define the kinematics of the state variables.
Next, we derive the conservative and dissipative forces by using the Energetic Variational Approach (EnVarA) \cite{Hyon}, \cite{giga2017}, \cite{Pliu} inspired by the work of Ericksen \cite{Ericksen} and combine them with Newton's force balance.
In the last step, we apply the laws of thermodynamics to the state functions and obtain the full model system.

We recall the following definitions from thermodynamics \cite{McQuarrie}, \cite{Salinas}.
\begin{description}
\item[Free energy:] The free energy $\Psi$ is a thermodynamic function depending on the state variables. The change in the free energy is the maximum amount of work that a system can perform. 
\item[Entropy:] The entropy given by $s=-\partial_\theta\Psi$ is an extensive state function. By the second law of thermodynamics the entropy of an isolated system increases and tends to the equilibrium state.
\item[Internal energy:] The internal energy $e=\Psi+s\theta$ is an extensive state function. It is a thermodynamic potential that can be analyzed in terms of microscopic motions and forces.
\end{description}

In addition to the state functions we recall the laws of thermodynamics \cite{Berry}.\\
The first law of thermodynamics relates the change in the internal energy with dissipation and heat
\begin{align}\label{1st law}
\frac{d\, e}{dt}= \nabla\cdot(\Sigma\cdot u)-\nabla \cdot q,
\end{align}
where $\Sigma$ denotes the stress tensor of the material and $u$ its velocity; this part expresses the work done by the system; and where $q$ denotes the heat flux.
We note that every total derivative can be written as follows
\begin{align}\label{2nd law}
\frac{d\, s}{dt} = \nabla \cdot j +\Delta,
\end{align}
where in case for the entropy $j$ denotes the entropy flux and $\delta$ is the entropy production rate.
The second law of thermodynamics states that the entropy production rate is non-negative:
\begin{align}
\Delta\geq 0.
\end{align}
\section{Derivation}\label{derivation}

In this section we derive a thermodynamic consistent model for the chemical reaction-diffusion equation with temperature.
For more details on chemical reactions we refer to the book by \cite{Kondepudi} and for a general chemical reaction equation derived by the energetic variational approach we refer to \cite{Wang}.\\ 

We consider the chemical reaction
\begin{align*}
\alpha A+\beta B\rightleftharpoons \gamma C
\end{align*}
and denote the concentration of each species by $c_i$, where $i=A,B,C$.

The \textbf{kinematics} of the concentration $c_i$ for each species is given by 
\begin{align}\label{kinematics}
\partial_t c_i +\dv(c_i u_i)=r(c,\theta) ~~ (x,t)\in \Omega\times(0,T)~~ \text{ for  } i=A,B,C
\end{align}
where $u_i: \Omega\to \mathbb{R}^n$ is the effective microscopic velocity for the i-th species and $r(c,\theta)$ denotes the reaction rate and we assume that $r(c,\theta)=0$ at equilibrium, i.e. the concentration of $A$ and $B$ lost in the forward reaction equals the amount gained in the backward reaction and the same for the concentration of $C$.
In addition, we assume that $u$ satisfies the non-flux boundary condition
\begin{align}\label{non-flux boundary u}
u_i\cdot n=0 ~~(x,T)\in \partial \Omega\times(0,T).
\end{align}
Moreover, we assume that the temperature moves along the trajectories of the flow map.

For the \textbf{free energy} we have the following equation
\begin{align}\label{free energy}
\psi(c, \theta)=\sum_i\psi_i(c_i,\theta)=\sum_i k_i^c c_i\theta\ln c_i-k_i^\theta c_i\theta\ln\theta
\end{align}
where for each species we consider the free energy of the ideal gas and we set the stoichiometric numbers to be one.

From the free energy we obtain the following thermodynamic quantities.
The \textbf{entropy} is given by
\begin{align}\label{entropy}
s(c,\theta)=\sum_i s_i(c_i,\theta)=-\frac{\partial \psi}{\partial \theta}= -\sum_i c_i\big( k_i^c\ln c_i-k_i^\theta(\ln\theta+1)\big).
\end{align}

\begin{remark}
We note that the free energy is convex in the temperature variable $\theta$.
This allows us to apply the implicit function theorem and solve the entropy equation \eqref{entropy} for $\theta$, i.e $\theta=\theta(\phi,\nabla\phi,s)$.
\end{remark}

Next, we can define the \textbf{internal energy} as follows
\begin{align}\label{internal energy}\begin{split}
e(c,\theta)=\sum_i e_i(c_i,\theta)&:=\psi+\theta s= \psi-\psi_\theta\theta\\
 &=\sum_i k_i^\theta c_i \theta =:e_1(c,s)	\end{split}
\end{align}
where we have used the convexity of the free energy $\psi$ with respect to $\theta$ to write the internal energy in terms of $c$ and $s$.

Next, we define the \textbf{chemical potential} as
\begin{align}\label{chem potential}
\mu_i:= \partial_{c_i}\psi_i(c_i,\theta)=k_i^c\theta(\ln c_i+1)-k_i^\theta \theta\ln \theta.
\end{align} 
We observe that at equilibrium we have 
\begin{align}
\mu_A+\mu_B=\mu_C
\end{align}
and by using the definition of the chemical potential we obtain
\begin{align}
\ln\bigg(\frac{c_A^{k_A^c}c_B^{k_B^c}}{c_C^{k_C^c}}\bigg)=\ln\theta\big(k_A^\theta+k_B^\theta-k_C^\theta\big)-\big({k_A^c}+{k_B^c}-{k_C^c}\big)
\end{align}
and
\begin{align}
\frac{c_A^{k_A^c}c_B^{k_B^c}}{c_C^{k_C^c}}=\frac{\theta^{k_A^\theta+k_B^\theta-k_C^\theta}}{e^{k_A^c+{k_B^c}-k_C^c}}=:K_{eq}(\theta)
\end{align}
where $K_{eq}(\theta)$ is the equilibrium constant for a fixed temperature $\theta$ of the reaction equation.
\begin{remark}
The quantity $\mu_A+\mu_B-\mu_C$ is known as affinity of a chemical reaction, introduced by De Donder as a new state variable of the system. Its sign shows the direction of the the chemical reaction and can be considered as the driving force of the reaction.
\end{remark}

Now, we return to the chemical reaction and write it as the following system of ordinary differential equations.
\begin{align}
r=\frac{1}{\sigma_i}\frac{d}{dt}c_i,~~\text{ for }i=A,B,C
\end{align}
and $\sigma=(\alpha,\beta,-\gamma)$.
We observe that if we subtract two of the equations we end up with two constraints
\begin{align*}
\gamma\frac{d c_A}{dt}+\alpha\frac{d c_C}{dt}=0,~~\gamma\frac{d c_B}{dt}+\beta\frac{d c_C}{dt}=0
\end{align*}
and as a consequence we obtain
\begin{align*}
\gamma c_A+\alpha c_C=Z_0,~~\gamma c_B+\beta c_C=Z_1,
\end{align*}
where the constants $Z_0$ and $Z_1$ are obtained by the initial concentrations.
Thus we only have one independent free parameter left, which we will cal reaction coordinate $R(t)$ and we can write 
\begin{align}
c_i(t)=c_{i,0}-\sigma_i R(t),~~\text{ for } i=A,B,C.
\end{align}
Moreover we have that the reaction rate $r$ is given by $r=\partial_t R(t)=R_t(t)$. 

This allows us to rewrite the free energy in terms of the reaction coordinate and temperature, i.e
\begin{align}
\psi(R,\theta)=\sum_i\psi_i(c_i(R),\theta).
\end{align}
Next, we introduce the dissipation due to the reaction $\mathcal{D}(R,R_t)$.
Applying the principle of virtual work we obtain that
\begin{align}
\frac{\delta F(R,\theta)}{\delta R}=-\frac{D(R,R_t)}{R_t}
\end{align}
where
\begin{align*}
\frac{\delta F(R,\theta)}{\delta R}&=-\mu_A-\mu_B+\mu_C\\
&= \theta\ln\theta\big(k_A^\theta+k_B^\theta-k_C^\theta\big)-\theta\ln\bigg(\frac{c_A^{k_A^c}c_B^{k_B^c}}{c_C^{k_C^c}}\bigg)-\theta\big({k_A^c}+{k_B^c}-{k_C^c}\big)
\end{align*}
The law of mass action determines the choice of the dissipation function.
The general form of the dissipation in the reaction we consider is the following
\begin{align}
\mathcal{D}(R,R_t)=\eta_1(R,\theta)R_t\ln(\eta_2(R,\theta)R_t+1),
\end{align}
where $\eta_1$ and $\eta_2$ are positive functions in $R$ and $\theta$.
Details of the derivation can be found in e.g. \cite{Groot}.

In chemical reactions a linear response function is considered as a simplified function of the general dissipation term.
We obtain
\begin{align}
\mathcal{D}(R,R_t)=\eta(R,\theta)|R_t|^2
\end{align}
again with $\eta$ being a positive function.
Using the principle of virtual work with these two dissipation terms yields the following reaction rates
\begin{align}
r_1:=&R_t=\bigg(\frac{c_Ac_B}{c_C}\bigg)^{k^c}\frac{\theta^{k^\theta}}{\exp(k^c)}-1
\intertext{for the choice $\eta_1(R,\theta)=\theta $ and $\eta_2(R,\theta)=1$ which we can write as the usual law of mass action}
  r_1:=&R_t= k_f(c_c,\theta) c_Ac_B-k_r(c_C,\theta)c_C,
\end{align}
where
\begin{align*}
k_f\sim \frac{\theta^{k^\theta/k^c}}{c_C}~~\text{and } k_r\sim \frac{1}{c_C}.
\end{align*}
Similar, for the linear response theory we obtain
\begin{align}
r_2:=&R_t= k^c\ln\bigg(\frac{c_Ac_B}{c_C}\bigg)+k^\theta\ln\theta-k^c
\end{align}
where we assume that $k_i^c=k^c$ and $k_i^\theta=k^\theta$ for $i=A,B,C$.
The above observations can be summarized in the following ODE system, where the derivation of the temperature part can be found at the end of this section.\\

In addition to the reaction part we also consider a diffusion part in the concentration.
To this end we introduce the dissipation due to diffusion
\begin{align*}
\mathcal{D}^D=\sum_i \eta_i(c_i,\theta)u_i^2+\nu|\nabla u_i|^2.
\end{align*}

\begin{remark}
Note that the dissipation depends on both the velocity of the flow map $u$ and its gradient $\nabla u$. 
Thus the parameters in front of the two terms can be seen as an interpolation between a Darcy-type and a Stokes-type of dissipation.
\end{remark}

Applying the principle of virtual work for the concentration part we obtain
\begin{align}
\nabla P_i=\nabla\big(c_i\psi_{c_i}-\psi_i\big)=c_i\nabla \mu_i,
\end{align}
where $P_i$ denotes the pressure and has the form
\begin{align}
P_i=c_i\psi_{c_i}-\psi_i=k^cc_i\theta.
\end{align}
\begin{lemma}\label{lemma 0}
The pressure satisfies
\begin{align*}
\nabla P_i=c_i \nabla \psi_{c_i} +s\nabla \theta.
\end{align*}
\end{lemma}
\begin{proof}
From the definition of the pressure we have $P_i(c_i,\theta)=\psi_{c_i} c_i -\psi$ and thus we compute 
\begin{align*}
\nabla P_i(c_i,\theta)&=\nabla(\psi_{c_i} \rho -\psi)=c_i\nabla \psi_{c_i} +\psi_{c_i} \nabla c_i-\nabla \psi\\
&=c_i \nabla \psi_{c_i} +\psi_{c_i} \nabla c_i-\psi_{c_i}\nabla c_i-\psi_\theta\nabla \theta= c_i\nabla \psi_{c_i}+s\nabla \theta.
\end{align*}
\end{proof}

Next, we apply the MDL and compute the variation of the dissipation with respect to the microscopic velocity $u$.
This yields
\begin{align*}
\delta_u\frac{1}{2}\mathcal{D}^{tot}&=2\int_\Omega \eta_i(c_i,\theta)u_i\cdot\tilde{u}+\nu \nabla u_i\cdot\nabla\tilde{u} dx\\
&=2\int_\Omega \eta_i(c_i,\theta)u_i\cdot\tilde{u}-\nu \Delta\nabla u_i\cdot\tilde{u} dx
\end{align*}
and hence the dissipative forces are
\begin{align*}
f_{diss}= \eta_i(c_i,\theta)u_i-\nu \Delta u_i
\end{align*}
From the classical Newton's force law for the concentration we deduce that the sum of the conservative and dissipative forces equals the change in the momentum, i.e.
\begin{align*}
f_{cons}+f_{diss}=\frac{d}{dt}(c_iu_i)
\end{align*}
Thus we obtain
\begin{align}\label{force balance}
\nu \Delta u_i-\eta_i(c_i,\theta)u_i-\nabla P_i=\frac{d}{dt}(c_iu_i)=\partial_t(c_iu_i)+\dv(c_iu_i\otimes u_i)
\end{align}
\begin{remark}
This is the momentum equation for the compressible Navier-Stokes equation with the addition of a Brinkman-type contribution in the dissipation.
\end{remark}
Before taking a closer look at the laws of thermodynamics we provide to useful Lemmas.
\begin{lemma}\label{Lemma 1}
$e_{1,s}(c,s)=\partial_s e_{1}(c,s)=\theta(c,s)$.
\end{lemma}
\begin{proof}
Applying the chain rule to the left-hand side of the equation yields
\begin{align*}
\partial_s e_{1}(c,s)&=\partial_s\big[\psi(c,\theta(c,s))+\theta(c,s)s\big]\\
&=\psi_\theta \theta_s +\theta_s s+\theta(c,s)=\theta(c,s),
\end{align*}
where we used that $s=-\psi_\theta$.
\end{proof}

\begin{lemma}\label{Lemma 2}
$\psi_{i,\theta}(c_i,\theta(c_i,s))=e_{1_i,c_i}(c_i,s)$.
\end{lemma}
\begin{proof}
By the chain rule applied to $e_{1_i,c_i}(c_i,s)$ we have
\begin{align*}
\partial_{c_i} e_{1_i}(c_i,s)&=\partial_{c_i}\big[\psi_i(c_i,\theta(c,s))+\theta(c,s)s\big]\\
&=\psi_{c_i}+\psi_\theta \theta_{c_i} +\theta_{c_i} s=\psi_\phi.
\end{align*}
\end{proof}

We note that we have a weak duality of the time evolution of the temperature and the total derivative of the entropy in the following way.
\begin{remark}\label{Remark 1}
If $\theta$ evolves as $\frac{d}{dt}\theta=\theta_t+u\cdot\nabla \theta$ then by testing this equation with the entropy $s$ in the weak form yields
\begin{align}\label{duality equation}
\int_\Omega \theta_t s +u\cdot\nabla \theta s\, dx=-\int_\Omega s_t \theta +\dv(su)\theta\, dx.
\end{align}
Thus $s$ satisfies $\frac{d}{dt}s=s_t+\dv(su)$.
\end{remark}
In the computations of the laws of thermodynamics we use the following constitutive relations and assumptions
\begin{itemize}
\item the Durhem equation $q=j\theta$;
\item Fourier's law $q=-\kappa \nabla \theta$;
\item the positivity of $\eta_i$, i.e $\eta_i(c_i,\theta)\geq 0$.
\end{itemize}

The general form of the first law of thermodynamics reads
\begin{align*}
\frac{d}{dt}\int_\Omega\big(K+e_1\big)=\text{work}+\text{heat},
\end{align*}
where in our case the kinetic energy $K=\sum_i c_i|u_i|^2$.
Then we compute 
\begin{align}
\frac{d}{dt}\int_\Omega e_1(c,s)dx &= \int_\Omega \big[\sum_i e_{1,c_i}c_{i,t} +e_{1,s}s_t\big] dx\\
\intertext{Using the kinematics for the density $c_i$ from equation (\ref{kinematics}) we obtain }\nonumber
&=\int_\Omega\big[ \sum_i e_{1,c_i}\big(-\nabla\cdot(c_i u_i)-\sigma_i R_t\big)+e_{1,s}s_t\big] dx\\
\intertext{Applying Lemma \ref{Lemma 2} yields}\nonumber
&=\int_\Omega\big[ -\sum_i\nabla \cdot\big(e_{1,c_i}c_i u_i\big)+ \sum_i\nabla \psi_{c_i}c_i\cdot u_i  -\sum_i \psi_{c_i}\sigma_i R_t+e_{1,s}s_t \big]dx\\
\intertext{In order to have the full expression for the gradient of the pressure we have to incorporate the term $s\nabla\theta$ which can only occur if the kinematics for the entropy are as in Remark \ref{Remark 1} and equation (\ref{duality equation}). Moreover by equation (\ref{2nd law}) we have}\nonumber
&=\int_\Omega\big[ -\nabla \cdot\big(\sum_i e_{1,c_i} c_i u_i+e_{1,s}su\big)+ \nabla \sum_i \psi_{c_i}c_i \cdot u_i+s\nabla e_{1,s}\cdot u\\\nonumber
&~~~~~ -\sum_i \psi_{c_i}\sigma_i R_t+e_{1,s}\big(\nabla \cdot j +\Delta\big) \big]dx\\
\intertext{By Lemma \ref{Lemma 1} and the Duhem equation we have}\nonumber
&=\int_\Omega\big[ -\nabla \cdot\big(\sum_i e_{1,c_i} c_i u_i+e_{1,s}su\big)+ \nabla \sum_i \psi_{c_i}c_i \cdot u_i+s\nabla e_{1,s}\cdot u\\\nonumber
&~~~~~ -\sum_i \psi_{c_i}\sigma_i R_t+\nabla \cdot q -\frac{q\cdot\nabla \theta}{\theta}+\theta\Delta \big]dx\\
\intertext{Now, we can apply Lemma \ref{lemma 0} to obtain}\nonumber
&=\int_\Omega \big[ -\nabla \cdot\big(\sum_i e_{1,c_i} c_i u_i+e_{1,s}su\big)+\nabla \cdot q +\sum_i \nabla(\psi_{c_i} \rho-c_i)\cdot u_i\\\nonumber
&~~~~~-\sum_i \psi_{c_i}\sigma_i R_t-\frac{q\cdot\nabla \theta}{\theta}+\theta\Delta \big]dx\\
\intertext{From the definition of the pressure and the absence of external forces and heat sources we have that}\nonumber
&=\int_\Omega\big[ \sum_i\nabla P_i\cdot u_i -\sum_i \psi_{c_i}\sigma_i R_t-\frac{q\cdot\nabla \theta}{\theta}+\theta\Delta\big] dx
\intertext{where we used that the divergence terms equal to zero under the boundary conditions $u\cdot n=0$ and $\nabla \theta\cdot n=0$. Thus we have}\nonumber
&= \int_\Omega \sum_i\big(\nu\Delta u_i -\eta(c_i,\theta) u_i+\partial_t(c_iu_i)+\dv(c_iu_i\otimes u_i) \big)\cdot u_idx\\ \nonumber
&~~~ -\int_\Omega \sum_i \mu_i\sigma_i R_t-\frac{q\cdot\nabla \theta}{\theta}+\theta\Delta dx\\
\intertext{and integration by parts yields}\begin{split}
&=\int_\Omega \big[\sum_i\big(-\nu|\nabla u_i|^2 -\eta(c_i,\theta)u_i^2-\sigma_iR_t|u_i|^2  -\mu_i\sigma_i R_t\big)\big]dx\\
&~~~ -\int_\Omega \big[\frac{q\cdot\nabla \theta}{\theta}+\theta\Delta\big]dx\\
&~~~ -\frac{1}{2}\sum_i\bigg(\frac{d}{dt}\int_\Omega c_i|u_i|^2 dx+\int_\Omega \dv(c_i|u_i|^2 u_i) dx\bigg)\end{split}
\end{align}
where we used the reaction equation and the momentum equation to express the pressure term.
Since there are no external forces or heat sources in our system the total internal energy must be conserved  and we obtain that
\begin{align}
\Delta= \frac{1}{\theta}\bigg(\sum_i\nu |\nabla u_i|^2+\sum_i\big(\sigma_iR_t+ \eta(c_i,\theta)\big) |u_i|^2 +\sum_i \mu_i\sigma_i R_t+ \frac{\kappa|\nabla \theta|^2}{\theta}\bigg).
\end{align}
We observe that we have to restrict the function $\eta_i(c_i,\theta)$ and the reaction rate $R_t$ such that 
\begin{align*}
\sum_i\big(\sigma_iR_t+ \eta_i(c_i,\theta)\big) |u_i|^2\geq 0.
\end{align*}
Thus, we note that the second law of thermodynamics $\Delta\geq 0$ is satisfied as long as $\theta>0$.

In addition, we have shown that the total energy,i.e. the sum of the kinetic energy and internal energy is conserved
\begin{align}\label{internal energy conserved}
\frac{d}{dt}\int_\Omega K(c,u)+ e_1(c,s) dx =\int_\Omega \text{ work }+\text{ heat }dx=0
\end{align}
since we assume that there are no external forces and no heat flux through the boundary.

Moreover, we have that the total entropy is increasing in time, i.e.
\begin{align}\label{entropy increase}
\frac{d}{dt}\int_\Omega s(c,\theta) dx=\int_\Omega s_t+\dv(su)dx=\int_\Omega \dv j +\Delta\geq 0, 
\end{align}
where we assume that there is no entropy flux through the boundary.\\

The above derivation can be summarized in the following general model for the chemical reaction with temperature
\begin{align}
&\partial_t c_i +\dv(c_i u_i)=-\sigma R_t,~~\text{ for } i=A,B,C\\
&\partial_t(c_iu_i)+\dv(c_iu_i\otimes u_i)-\nu \Delta u_i+\eta_i(c_i,\theta)u_i=k^c\nabla c_i\theta\\
&\partial_t s +\dv(\sum_i s_i u_i)=\dv\bigg(\frac{\kappa\nabla \theta}{\theta}\bigg)+\Delta
\end{align}
where
\begin{align}
&R_t=r_1=k_f(c_c,\theta) c_Ac_B-k_r(c_C,\theta)c_C
\intertext{for the general law of mass action or}
&R_t=r_2=k^c\ln\bigg(\frac{c_Ac_B}{c_C}\bigg)+k^\theta\ln\theta-k^c
\intertext{for the linear response theory. In addition, we have that the entropy production rate is given by}
&\Delta= \frac{1}{\theta}\bigg(\sum_i\nu |\nabla u_i|^2+\sum_i\big(\sigma_iR_t+ \eta(c_i,\theta)\big) |u_i|^2 +\sum_i \mu_i\sigma_i R_t+ \frac{\kappa|\nabla \theta|^2}{\theta}\bigg)
\intertext{where the chemical potential is defined as}
&\mu_i=k^c\theta(\ln c_i+1)-k^\theta \theta\ln \theta
\intertext{and the entropy is defined by}
&s=\sum_i s_i= -\sum_i c_i\big( k^c\ln c_i-k^\theta(\ln\theta+1)\big)
\end{align}
After deriving the general model for the reaction-diffusion equation with temperature we consider a simplified version.
To this end, we make several assumptions.
\begin{itemize}
\item First, we assume that the dissipation $\mathcal{D}$ depends only on the velocity $u$ and not on its derivative, i.e. the dissipation we consider is of Darcy type.
\item Second, we assume that Newton,s force law reduces to a force balance between conservative and dissipative forces, i.e. $f_{cons}+f_{diss}=0$.
This yields a Darcy type law for the velocity $\eta(c_i,\theta)u_i=-\nabla P_i$.
\item Finally,as a consequence of the above we assume that we can neglect the influence of the kinetic energy and set it equal to zero. 
Thus the conservation of internal energy holds $\frac{d}{dt}\int_\Omega e_1(\rho,s)dx =0$.
\end{itemize}
Hence, we obtain the reaction-diffusion equation with temperature for a Darcy type velocity.
\begin{align}
&\partial_t c_i +\dv(c_i u_i)=-\sigma R_t,~~\text{ for } i=A,B,C\\
&\eta_i(c_i,\theta)u_i=-k^c\nabla c_i\theta\\\label{entropy eq simple}
&\partial_t s +\dv(\sum_i s_i u_i)=\dv\bigg(\frac{\kappa\nabla \theta}{\theta}\bigg)+\Delta
\end{align}
where
\begin{align*}
&R_t=r_1=k_f(c_c,\theta) c_Ac_B-k_r(c_C,\theta)c_C
\intertext{for the general law of mass action}
&R_t=r_2=k^c\ln\bigg(\frac{c_Ac_B}{c_C}\bigg)+k^\theta\ln\theta-k^c
\intertext{and for the linear response theory. In addition, we have}
&\Delta= \frac{1}{\theta}\bigg(\sum_i\nu |\nabla u_i|^2+\sum_i\big(\sigma_iR_t+ \eta(c_i,\theta)\big) |u_i|^2 +\sum_i \mu_i\sigma_i R_t+ \frac{\kappa|\nabla \theta|^2}{\theta}\bigg)\\
&\mu_i=k^c\theta(\ln c_i+1)-k^\theta \theta\ln \theta\\
&s=\sum_i s_i= -\sum_i c_i\big( k^c\ln c_i-k^\theta(\ln\theta+1)\big)
\end{align*}
This system of equations can be written in a condensed form by eliminating the velocity in the reaction and entropy equation.
Moreover we take a closer look at the temperature.
To this end, we explicitly compute the left-hand side of equation \eqref{entropy eq simple}.
\begin{align}
\partial_t s= -\sum_i \big(k^c \partial_t c_i\ln c_i +k^c\partial_t c_i-k^\theta \partial c_i (\ln \theta +1)-k^\theta c_i\frac{\partial_t \theta}{\theta}\big)
\end{align}
and
\begin{align}
\nonumber \dv(\sum_i s_i u_i)&= -\dv\bigg(\sum_i c_i\big( k^c\ln c_i-k^\theta(\ln\theta+1)\big)u_i\bigg)\\
&=-\sum_i\bigg[k^c\ln c_i\dv\big(c_i u)+k^c u_i\nabla c_i-k^\theta\big(\ln\theta+1 \big)\dv\big(c_iu_i\big)-k^\theta c_iu_i\frac{\nabla \theta}{\theta}\bigg]
\end{align}
Adding these two equations and using the reaction equation for the concentration we obtain
\begin{align*}
\partial_t s+\dv(\sum_i s_i u_i)=&\sum_i\big(k^c \ln c_i +k^c -k^\theta(\ln\theta +1)\big)\sigma_i R_t+\sum_ik^c c_i\dv u_i\\
&+\sum_i k^\theta\frac{c_i}{\theta}\big(\partial_t\theta+u_i\nabla \theta\big)
\end{align*}
Thus multiplying the entropy equation by $\theta$ yields
\begin{align*}
\theta \big(\partial_t s+\dv(\sum_i s_i u_i)\big)=& \sum_i k^\theta c_i(\partial_t\theta +u_i\nabla \theta)+\sum_ik^c c_i\theta \dv u_i\\
& +\sum_i (\mu_i+k^\theta\theta)\sigma_i R_t
\end{align*}
and the temperature equations reads
\begin{align}
\sum_i k^\theta c_i\big(\partial_t\theta +\dv(\theta u_i)\big)=\kappa\Delta\theta +\sum_i\sigma_i k^\theta\theta R_t +\sum_i\big(\nu|\nabla u_i|^2+\eta_i(c_i,\theta)|u_i|^2\big).
\end{align}
This yields the following system of equations for the reaction-diffusion system 
\begin{align}
&\partial_t c_i -k^c\Delta c_i=-\sigma_i R_t+k^c\nabla\cdot\big(c_i\nabla \ln\theta\big),~~\text{ for } i=A,B,C\\\begin{split}
&\sum_i k^\theta c_i\bigg[\partial_t\theta -k^\theta\bigg(\frac{\nabla c_i\cdot\nabla \theta}{c_i}+\frac{|\nabla \theta|^2}{\theta}\bigg)\bigg]=\kappa\Delta\theta +\sum_i\sigma_i k^\theta\theta R_t\\
 &~~~~~~+ (k^c)^2\sum_i \bigg[(\eta_i-1)\frac{|\nabla(c_i\theta)|^2}{c_i\theta}+\Delta(c_i\theta)\bigg]\end{split}
\end{align}
where we have the two different reaction rates derived from the general law of mass action and the linear response theory
\begin{align*}
&R_t=r_1= k_f(c_c,\theta) c_Ac_B-k_r(c_C,\theta)c_C,\\
&R_t=r_2=k^c\ln\bigg(\frac{c_Ac_B}{c_C}\bigg)+k^\theta\ln\theta-k^c.
\end{align*}

\section{Besov Spaces}
In this section we will present the theory behind the well-posedness problem for the reaction-diffusion system with temperature.
In order to so, we introduce the Besov spaces by using the Littlewood-Paley decomposition.
For the details in the Theorems and Definitions presented in this section, we refer to  \cite{Bahouri} and \cite{Sawano}. \\

We first define the building blocks of the theory of Besov spaces, the dyadic partition of unity.
Let $\mathcal{C}$ be the annulus $\{\xi\in \mathbb{R}^d~:~3/4\leq |\xi|\leq 8/3\}$, and let $\phi$ be a radial function with values in the interval $[0,1]$ belonging to the space $\mathcal{D}(\mathcal{C})$ with the following partition of unity
\begin{align*}
\forall \xi \in \mathbb{R}^d\setminus\{0\},~\sum_{j\in \mathbb{Z}}\phi\big(2^{-j}\xi\big)=1.
\end{align*}
We observe that for $|j-i|\geq 2$ we have $
\supp \phi\big(2^{-j}\cdot\big)\cap \supp \phi\big(2^{-i}\cdot\big)=\emptyset$.
In addition, we define the Fourier transform $\mathcal{F}$ of the whole space $\mathbb{R}^d$.
Then we can define the homogeneous dyadic block $\dot{\Delta}_j$ and the homogeneous low-frequency cut-off operator $\dot{S}_j$ for all $j$
\begin{align*}
\dot{\Delta}_j u&=\mathcal{F}^{-1}\big(\phi(2^{-j}\xi)\mathcal{F}u\big)\\
\dot{S}_j u&= \sum_{i\leq j-1} \dot{\Delta}_i u.
\end{align*}
Hence, we can write the formal Littlewood-Paley decomposition
\begin{align*}
\text{Id}=\sum_j \dot{\Delta}_j.
\end{align*}
This allows us to define the homogeneous Besov spaces.
\begin{definition}
The homogeneous Besov spaces $\dot{B}^s_{p,r}$ with $s\in \mathbb{R}$, $p,r\in [1,\infty]^2$ and 
\begin{align*}
s<\frac{d}{2} \text { if } r>1, \quad s \leq \frac{d}{2} \quad \text { if } \quad r=1
\end{align*}
consist of all homogeneous tempered distributions $u$ such that 
\begin{align*}
\|u\|_{\dot{B}^s_{p,r}}:=\bigg(\sum_{j\in\mathbb{Z}}2^{rjs}\|\dot{\Delta}_j u\|_{L^p}^r\bigg)^{1/r}<\infty.
\end{align*}
\end{definition}
 We remark that the (semi-)norms $\|\cdot\|_{\dot{H}^s}$ and $\|\cdot\|_{\dot{B}^s_{2,2,}}$ are equivalent.
Furthermore, we observe that $\dot{H}^s\subset \dot{B}^s_{2,2}$ and equality holds if $s<d/2$.\\
We have have the following remark
\begin{remark}
Let $(s_1,s_2)\in \mathbb{R}^2$ and $1\leq p_1,p_2,r_1,r_2\leq \infty$ with $s<d/p$ or $s=d/p$ if $r=1$. 
Then the space $\dot{B}_{p_1,r_1}^{s_1}\cap \dot{B}_{p_2,r_2}^{s_2}$ is endowed with the norm $\|\cdot\|_{\dot{B}_{p_1,r_1}^{s_1}}+\|\cdot\|_{\dot{B}_{p_2,r_2}^{s_2}}$ is a complete normed space.
\end{remark}
One special feature of homogeneous Besov spaces is there scaling property.
Next, we have some useful embeddings.
\begin{prop}
For $p\in [1,\infty)$ the space $ \dot{B}_{p,1}^{d/p}$ is continuously embedded in the space $C^0$, i.e. the space of continuous functions vanishing at infinity.
\end{prop}
\begin{prop}
Let $1\leq p_1\leq p_2\leq \infty$ and let $1\leq r_1\leq r_2\leq \infty$. Then for any $s\in \mathbb{R}$ the space $\dot{B}_{p_1,r_1}^{s} $ is continuously embedded in $\dot{B}_{p_2,r_2}^{s-d(1/p_1-1/p_2)} $.
\end{prop}
\begin{remark}
From this point on we work with the Besov spaces $ \dot{B}_{2,1}^{s}$ and by the above Proposition we have that it is continuously embedded into $\dot{H}^s$.
\end{remark}

The following product rule is the key in the well-posedness result for the reaction-diffusion system. 
\begin{prop}
Let $u\in \dot{B}_{2,1}^{s_1}$ and let $v\in  \dot{B}_{2,1}^{s_2}$  with $s_1,\,s_2\leq d/2$.
If $s_1+s_2>0$ then the product $uv$ belongs to $ \dot{B}_{2,1}^{s_1+s_2-d/2}$ and the following inequality holds
\begin{align*}
\|uv\|_{ \dot{B}_{2,1}^{s_1+s_2-d/2}}\leq C \|u\|_{ \dot{B}_{2,1}^{s_1}}\|v\|_{ \dot{B}_{2,1}^{s_2}},
\end{align*}
where the constant $C$ depends on $s_1,\, s_2$ and the dimension $d$.
\end{prop}
We observe that for $s=d/2$ fixed we obtain an algebra structure for the space $\dot{B}_{2,1}^{d/2}$, i.e.
\begin{align*}
\dot{B}_{2,1}^{d/2}\times \dot{B}_{2,1}^{d/2}\to \dot{B}_{2,1}^{d/2}.
\end{align*}
Next, we define the time-space Besov spaces, where the idea is to bound each dyadic block in $L^q\big([0,T];L^p\big)$ than to estimate directly the solution of the whole partial differential equation in $L^q\big([0,T];\dot{B}_{p,r}^s\big)$.
\begin{definition}
For $T>0$ and $s\in \mathbb{R}$ let $1\leq r,p\leq \infty$ and let the assumptions of Definition 3.1 hold.
Then we set
\begin{align*}
\|u\|_{\mathcal{L}^q_T\big(\dot{B}_{p,r}^s\big)}=\bigg(\sum_{j\in\mathbb{Z}} 2^{rjs}\|\dot{\Delta}_j u\|_{L_T^q\big( L^{p} \big)}\bigg)^{1/r}.
\end{align*}
\end{definition}
The spaces $\mathcal{L}^q_T\big(\dot{B}_{p,r}^s\big)$ can be linked with the more classical spaces $L^q\big([0,T];\dot{B}_{p,r}^s\big)$ via the Minkowski inequality and we obtain
\begin{align*}
\|u\|_{\mathcal{L}^q_T\big(\dot{B}_{p,r}^s\big)}\leq \|u\|_{L^q\big([0,T];\dot{B}_{p,r}^s\big)}~~\text{if } r\geq p,\\
\intertext{and}
\|u\|_{\mathcal{L}^q_T\big(\dot{B}_{p,r}^s\big)}\leq \|u\|_{L^q\big([0,T];\dot{B}_{p,r}^s\big)}~~\text{if } r\leq p.
\end{align*}
\begin{remark}
The general principles is that all properties of continuity of the product, composition, etc. remain true in these time-space Besov spaces too.
The exponent $q$ just has to behave according to H\"older's inequality for the time variable. 
\end{remark}
The following result is the key in the existence proof later on.
\begin{theorem}\label{Besov Theorem}
Let $u_0\in \dot{B}_{2,1}^{s}$ be the initial data with regularity $s\leq d/2$.
In addition, let $f\in \mathcal{L}^1_T\big(\dot{B}_{2,1}^{s}\big)$ be the driving force, and we denote by $u$ the unique solution to the following linear parabolic PDE
\begin{align}
&\partial_t u-\Lambda u=f~~~\text{in } \mathbb{R}_+\times \mathbb{R}^d,\\
&u\big|_{t=0}=u_0~~~\text{in } \mathbb{R}^d,
\end{align}
where $\Lambda$ is a linear second order strongly elliptic operator.
Then the solution $u$ belongs to the space $\mathcal{L}_T^\infty\big(\dot{B}_{2,1}^{s}\big)$ and the pair $\big(\partial_t u,\Delta u\big)$ to $\mathcal{L}_T^1\big(\dot{B}_{2,1}^{s}\big)$.
Furthermore the following inequality holds
\begin{align*}
\|u\|_{\mathcal{L}_T^\infty(\dot{B}_{2,1}^{s})}+\|\partial_t u\|_{\mathcal{L}_T^1(\dot{B}_{2,1}^{s})}+\|\Delta u\|_{\mathcal{L}_T^1(\dot{B}_{2,1}^{s})}\leq C\big[\|u_0\|_{\dot{B}_{2,1}^{s}}+\|f\|_{\mathcal{L}_T^1(\dot{B}_{2,1}^{s})}\big].
\end{align*} 
\end{theorem}
In addition, the following Corollary is used frequently in the later part.
\begin{corollary}
Let $1\leq q,r\leq \infty$, $2\leq p<\infty$ and $s\in\mathbb{R}$ and let $I=[0,T)$ for any $T>0$.
Suppose is a solution to the system ().
Then there exists a constant $C>0$ depending on $q,p,r,n$ such that
\begin{align*}
\|u\|_{\mathcal{L}^q_T\big(\dot{B}_{p,r}^{s+2/q}\big)}\leq C\big(\|u\|_{\dot{B}_{p,r}^s}+\|f\|_{\mathcal{L}^1_T\big(\dot{B}_{p,r}^s\big)}
\end{align*}
for $0<T\leq \infty$.
\end{corollary}

The following result considers the action of smooth functions on the Besov space $\dot{B}_{2,1}^{d/2}$.
\begin{lemma}\label{Besov Lemma}
Let $f$ be a smooth function on $\mathbb{R}$ which vanishes at $0$.
Then for any function $u\in \dot{B}_{2,1}^{d/2}$ the function $f(u)$ is still element of $\dot{B}_{2,1}^{d/2}$ and the following inequality holds
\begin{align*}
\|f(u)\|_{\dot{B}_{2,1}^{d/2}}\leq Q\big(f, \|u\|_{L^\infty}\big)\|u\|_{\dot{B}_{2,1}^{d/2}},
\end{align*}
where $Q$ is a smooth function depending on the value of $f$ and its derivative.
\end{lemma}
The above Lemma can also be applied to a product of two functions in the following way.
\begin{corollary}
Let $u\in \dot{B}_{2,1}^{d/2}$ and $v\in \dot{B}_{2,1}^s$ such that the product is continuous in $\dot{B}_{2,1}^{d/2}\times	\dot{B}_{2,1}^s\to \dot{B}_{2,1}^s$.
Let$f$ be a smooth function on $\mathbb{R}$, then $f(u)v\in \dot{B}_{2,1}^s$ and the following inequality holds
\begin{align*}
\|f(u)v\|_{\dot{B}_{2,1}^s}\lesssim Q\big(f,\|u\|_{L^\infty}\big)\|u\|_{\dot{B}_{2,1}^{d/2}}\|v\|_{\dot{B}_{2,1}^s}.
\end{align*}
\end{corollary}

\section{Well-Posedness Result}
Now, we have all the necessary tools together to show the existence of solutions.
We recall the Darcy-type model for which we introduce perturbations close to equilibrium, where we set $c_i(t,x)$ for $i=A,B,C$ to be the concentration of the i-th species and $\theta(t,x)$ the temperature of the system for $(t,x)\in [0,T]\times \mathbb{R}^d$ for $d=2,3$.
The system then reads 
\begin{align}\label{R-D system 1}
&\partial_t c_i -k^c\Delta c_i=-\sigma_i R_t+k^c\nabla\cdot\big(c_i\nabla \ln\theta\big),~~\text{ for } i=A,B,C\\\label{R-D system 2}\begin{split}
&\sum_i k^\theta c_i\bigg[\partial_t\theta -k^\theta\bigg(\frac{\nabla c_i\cdot\nabla \theta}{c_i}+\frac{|\nabla \theta|^2}{\theta}\bigg)\bigg]=\kappa\Delta\theta +\sum_i\sigma_i k^\theta\theta R_t\\
 &~~~~+ (k^c)^2\sum_i \bigg[(\eta_i-1)\frac{|\nabla(c_i\theta)|^2}{c_i\theta}+\Delta(c_i\theta)\bigg]\end{split}
\end{align}
with 
\begin{align*}
R_t= k^c\ln\bigg(\frac{c_Ac_B}{c_C}\bigg)-k^\theta\ln\theta+k^c
\end{align*} 
\begin{remark}
The equilibrium state is defined such that $R_t(\tilde{c}_A,\tilde{c}_B,\tilde{c}_C,\tilde{\theta})=0$, where we observe that if $(\tilde{c}_A,\tilde{c}_B,\tilde{c}_C,\tilde{\theta})$ is at equilibrium then $(\lambda\tilde{c}_A,\lambda\tilde{c}_B,\lambda\tilde{c}_C,\lambda^{k^c/k^\theta}\tilde{\theta})$ is also an equilibrium state.
Thus, we can assume that without loss of generality $\tilde{c}_i\geq 1/h^2$ for $i=A,B,C$ and $\tilde{\theta}\geq 1/h^2$ for any $0<h<1$.
\end{remark}
Next, we rewrite the system as perturbation to the equilibrium state $(\tilde{c}_A,\tilde{c}_B,\tilde{c}_C,\tilde{\theta})$ by setting
\begin{align*}
c_i=\tilde{c}_i+z_i~~\text{for } i=A,B,C~~\text{and } \theta=\tilde{\theta}+\omega.
\end{align*}
In the nest step we linearize the reaction rate $R_t$ by doing a first order Taylor expansion around the equilibrium state $R_t=0$ and obtain
\begin{align*}
R_t=r=k^c\sum_{j}\sigma_j\frac{z_j}{\tilde{c}_j}-k^\theta\frac{\omega}{\tilde{\theta}}
\end{align*}
The perturbed system now reads
\begin{align}\label{Perturbed system 1}\begin{split}
&\partial_t z_i -k^c\Delta z_i=-\sigma_i\bigg[ k^c\sum_{j}\sigma_j\frac{z_j}{\tilde{c}_j}-k^\theta\frac{\omega}{\tilde{\theta}}\bigg]\\
& ~~~~~~+k^c\bigg[\nabla z_i^k\cdot \frac{\nabla \omega}{\omega+\tilde{\theta}}+z_i\frac{\Delta \omega}{\omega+\tilde{\theta}}-(z_i+\tilde{c}_i)\frac{|\nabla \omega|^2}{(\omega+\tilde{\theta})^2}\bigg],\end{split}
\intertext{ for $i=A,B,C$}\label{Perturbed system 2}\begin{split}
&\sum_i k^\theta (z_i+\tilde{c}_i)\bigg[\partial_t\omega -k^\theta\bigg(\frac{\nabla z_i\cdot\nabla \omega}{z_i+\tilde{c}_i}+\frac{|\nabla \omega|^2}{\omega+\tilde{\theta}}\bigg)\bigg]=\kappa\Delta\omega\\
&~~~ +\sum_i\sigma_i k^\theta(\omega+\tilde{\theta})\big(k^c\sum_{j}\sigma_j\frac{z_j}{\tilde{c}_j}-k^\theta\frac{\omega}{\tilde{\theta}}\big) \\
 &~~~+ (k^c)^2\sum_i \bigg[(\eta_i-1)\frac{|(z+_i+\tilde{c}_i)\nabla \omega+(\omega+\tilde{\theta})\nabla z_i|^2}{(z_i+\tilde{c}_i)(\omega+\tilde{\theta})}\bigg]\\
& ~~~+(k^c)^2\sum_i \bigg[(z_i+\tilde{c}_i)\Delta\omega+\nabla z_i\cdot\nabla\omega+\omega \Delta z_i\bigg]\end{split}
\end{align}
\begin{remark}
We note that we modified the concentration equation and temperature equation slightly by subtracting the term $\tilde{c}_i \Delta \omega$ and $\tilde{\theta}\Delta z_i$ respectively.
This regularization of the equations ensures that for constant concentration or constant temperature, i.e. the perturbation of the concentration $z_i=0$ and perturbation of the temperature $\omega=0$, we obtain that the perturbation in the state variables goes to zero, and thus the system returns to equilibrium.
\end{remark}
We now state the well-posed result for the reaction-diffusion system with temperature
\begin{theorem}[Well-Posedness for the R-D System with Temperature]\label{main Theorem}~\\
Let there be a small positive number $h>0$ and let the initial data satisfy the following condition 
\[c_{i,0}-\tilde{c}_i=z_{i,0}\in \dot{B}_{2,1}^{d/2}~\text{ for }~~i=A,B,C~~\text{ and }~~\theta_0-\tilde{\theta}=\omega_0\in \dot{B}_{2,1}^{d/2}\]
 and let the initial data fulfill the smallness condition
\begin{align}
\sum_i \|z_{i,0}\|_{\dot{B}_{2,1}^{d/2}}+\|\omega_0\|_{\dot{B}_{2,1}^{d/2}}\leq h^4.
\end{align}
Then the reaction-diffusion system with temperature close to equilibrium admits a unique global-in-time strong solution belonging to the following function spaces
\begin{align}
c_i-\tilde{c}_i&=:z_i\in \mathcal{L}_T^\infty\big(\dot{B}_{2,1}^{d/2}\big)~~\text{and }~~ \partial_t c_i,\Delta c_i\in \mathcal{L}_T^1\big(\dot{B}_{2,1}^{d/2}\big)~~ \text{for } i=A,B,C\\
\theta-\tilde{\theta}&=:\omega\in \mathcal{L}_T^\infty\big(\dot{B}_{2,1}^{d/2}\big)~~\text{and }~~ \partial_t \theta,\Delta \theta\in \mathcal{L}_T^1\big(\dot{B}_{2,1}^{d/2}\big).
\end{align}
In addition, the solution satisfies the following the inequality
\begin{align}
\sum_i\|c_i-\tilde{c}_i\|_{\mathcal{B}}+\|\theta-\tilde{\theta}\|_{\mathcal{B}}\leq h^2,
\end{align}
where we define the space $\mathcal{B}$ is defined as follows
\begin{align}\label{B norm def}
\|u\|_{\mathcal{B}}:=\|u\|_{ \mathcal{L}_T^\infty\big(\dot{B}_{2,1}^{d/2}\big)}+\|\partial_t u\|_{ \mathcal{L}_T^1\big(\dot{B}_{2,1}^{d/2}\big)}+\|\Delta u\|_{ \mathcal{L}_T^1\big(\dot{B}_{2,1}^{d/2}\big)}.
\end{align}
\end{theorem}
The idea of the proof is to construct an iterative scheme of the following form
\begin{align*}
\partial_t f^{k+1}+ \Lambda f^{k+1}= F^{k}
\end{align*}
where we show that this yields a bounded sequence in some Besov space and where the difference between two iterations form a null sequence. 
From this we can follow that the iterative sequence convergences.
\subsection{Proof of Theorem \ref{main Theorem}}
As mentioned before, the idea of the proof of the theorem is to use an approximate scheme to construct the solution to the perturbed system of equations \eqref{Perturbed system 1}-\eqref{Perturbed system 2}.
We set the first term in the sequence $(z_i^0(t,x),\omega^0(t,x))$ is set to be zero everywhere in $\mathbb{R}_+\times \mathbb{R}^d$.
Then, we set $(z_i^k(t,x),\omega^k(t,x))$ to be the solution of the following linear approximate system.
\begin{align}
\partial_t z_i^{k+1} -k^c\Delta z_i^{k+1}&= F_i^k~~\text{for } i=A,B,C\\
\partial_t \omega^{k+1}-\bigg(\frac{(k^c)^2}{k^\theta}+\frac{\kappa}{k^\theta\sum_i \tilde{c}_i}\bigg)\Delta \omega^{k+1}&= G^k
\end{align}
where
\begin{align}\begin{split}
F_i^k&= -\sigma_i\bigg[ k^c\sum_{j}\sigma_j\frac{z^k_j}{\tilde{c}_j}-k^\theta\frac{\omega^k}{\tilde{\theta}}\bigg] +k^c \bigg(\frac{1}{\tilde{\theta}}+f(\omega^k)\bigg)\nabla z_i^k\cdot\nabla \omega^k\\
&~~~ +k^c\bigg[z_i^{k}(\frac{1}{\tilde{\theta}}+f(\omega^k))\Delta \omega^k-(z_i^{k}+\tilde{c}_i)(\frac{1}{\tilde{\theta}^2}+g(\omega^k))|\nabla \omega^k|^2\bigg)\bigg] \end{split} \\ \begin{split}
G^k&= k^\theta \sum_i \nabla z_i^k\cdot\nabla \omega^k \bigg(\frac{1}{\sum_i \tilde{c}_i}+f(\sum_i z_i^k)\bigg)+k^\theta|\nabla \omega^k|^2\bigg(\frac{1}{\tilde{\theta}}+f(\omega^k)\bigg)+\kappa f(c)\Delta\omega^k\\
&~~~ +\bigg(\frac{1}{\sum_i k^\theta \tilde{c}_i}+f(\sum_i z_i^k)\bigg)\sum_i\sigma_i k^\theta(\omega^k+\tilde{\theta})\big(k^c\sum_{j}\sigma_j\frac{z_j^k}{\tilde{c}_j}-k^\theta\frac{\omega^k}{\tilde{\theta}}\big) \\
 &~~~+ \bigg(\frac{1}{\sum_i k^\theta \tilde{c}_i}+f(\sum_i z_i^k)\bigg)(k^c)^2\sum_i \bigg[(\eta_i-1)\frac{|(z_i^k+\tilde{c}_i)\nabla \omega^k+(\omega^k+\tilde{\theta})\nabla z_i^k|^2}{(z_i^k+\tilde{c}_i)(\omega^k+\tilde{\theta})}\bigg]\\
&~~~ + \bigg(\frac{1}{\sum_i k^\theta \tilde{c}_i}+f(\sum_i z_i^k)\bigg)(k^c)^2\sum_i \bigg[\nabla z_i^k\cdot\nabla\omega^k+\omega^k \Delta z_i^k\bigg]\end{split}
\end{align}
and where we define
\begin{align*}
f(x):=\frac{1}{\tilde{x}+x}-\frac{1}{\tilde{x}}~~ \text{and } g(x):=\frac{1}{(x+\tilde{x})^2}-\frac{1}{\tilde{x}^2}
\end{align*}
We note that for $x >-\tilde{x}$ $f$ and $g$ are smooth functions and in addition for $|x|/\tilde{x}\ll 1$ the function $f$ is $\mathcal{O}(x)$ and  $g$ is $\mathcal{O}(x^2)$ respectively.

\begin{prop}[Iterative scheme]\label{existence prop}
Let  $(z_A^k,z_B^k,z_C^k,\omega^k)$ be a unique global-in-time classical solution to the perturbed system \eqref{Perturbed system 1}-\eqref{Perturbed system 2}.
Then the solution belongs to the space $\mathcal{L}_T^\infty\big(\dot{B}_{2,1}^{d/2}\big)$ fulfilling the following inequalities
\begin{align}\label{est 1}
&\|z_i^k\|_{\mathcal{B}}\leq h^2~~\text{for } i=A,B,C~~\|\omega^k\|_{\mathcal{B}}\leq h^2.
\end{align} 
Furthermore, the difference between two consecutive solutions satisfies
\begin{align}\label{est 2}
&\|\delta z_i^k\|_{\mathcal{B}}\leq h^2~~\text{for } i=A,B,C~~\|\delta\omega^k\|_{\mathcal{B}}\leq h^2.
\end{align}
\end{prop}

From this proposition the proof of Theorem \ref{main Theorem} can be proven as follows.
Let $(z_A^k,z_B^k,z_C^k,\omega^k)$ be an approximate solution satisfying the estimate of Proposition \ref{existence prop}.
Then the following series converges
\begin{align*}
\sum_{k=1}^\infty \sum_i\|\delta z_i^k\|_{\mathcal{B}}+\|\delta\omega^k\|_{\mathcal{B}}<\infty.
\end{align*}
Thus we conclude that the sequence $\big(z_A^k,z_B^k,z_C^k,\omega^k\big)_{k\in \mathbb{N}}$ forms a Cauchy sequence in the space $\mathcal{B}$ and the limit $(z_A,z_B,z_C,\omega)$ is a strong solution of the perturbed system \eqref{Perturbed system 1}-\eqref{Perturbed system 2}.\\

The proof of this proposition is split up into several steps.
The first one is to show the approximate solutions are bounded in the Besov space $\mathcal{B}$.\\

\textit{Concentration equation}:
We consider an approximate solution $z_i^{k}$ and aim to show that the next level in the approximation is bounded by $\|z_i^{k+1}\|_{\mathcal{B}}\leq h^2$.
Then, by Theorem \ref{Besov Theorem} we have that the norm of $\|z_i^{k+1}\|_{\mathcal{B}}$ is bounded by
\begin{align*}
\|z_i^{k+1}\|_{\mathcal{L}_T^\infty(\dot{B}_{2,1}^{d/2})}&+\|\partial_t z_i^{k+1}\|_{\mathcal{L}_T^1(\dot{B}_{2,1}^{d/2})}+k^c\|\Delta z_i^{k+1}\|_{\mathcal{L}_T^1(\dot{B}_{2,1}^{d/2})}\\
&\leq C\big[\|z_{i,0}\|_{\dot{B}_{2,1}^{d/2}}+\|F_i^k\|_{\mathcal{L}_T^1(\dot{B}_{2,1}^{d/2})}\big]
\end{align*}
By the smallness assumption on the initial data we obtain 
\begin{align}\label{eq conc 1}\begin{split}
\|z_i^{k+1}\|_{\mathcal{L}_T^\infty(\dot{B}_{2,1}^{d/2})}&+\|\partial_t z_i^{k+1}\|_{\mathcal{L}_T^1(\dot{B}_{2,1}^{d/2})}+k^c\|\Delta z_i^{k+1}\|_{\mathcal{L}_T^1(\dot{B}_{2,1}^{d/2})}\\
&\leq C\big[h^4+\|F_i^k\|_{\mathcal{L}_T^1(\dot{B}_{2,1}^{d/2})}\big]\end{split}
\end{align}
We claim that the forcing term is bounded by
\begin{align*}
\|F_i^k\|_{\mathcal{L}_T^1(\dot{B}_{2,1}^{d/2})}&\lesssim k^c \sum_j \frac{\|z_j^k\|_{\mathcal{L}_T^1(\dot{B}_{2,1}^{d/2})}}{\tilde{c}_j} + k^\theta\frac{\|\omega^k\|_{\mathcal{L}_T^1(\dot{B}_{2,1}^{d/2})}}{\tilde{\theta}}\\
&+k^c\big(\frac{1}{\tilde{\theta}}+\|f(\omega^k)\|_{\mathcal{L}_T^\infty(\dot{B}_{2,1}^{d/2})}\big)\|\nabla z_i^k\|_{\mathcal{L}_T^2(\dot{B}_{2,1}^{d/2})}\|\nabla\omega^k\|_{\mathcal{L}_T^2(\dot{B}_{2,1}^{d/2})}\\
&+k^c \|z_i^k\|_{\mathcal{L}_T^\infty(\dot{B}_{2,1}^{d/2})}\big(\frac{1}{\tilde{\theta}}+\|f(\omega^k)\|_{\mathcal{L}_T^\infty(\dot{B}_{2,1}^{d/2})}\big)\|\Delta \omega^k\|_{\mathcal{L}_T^1(\dot{B}_{2,1}^{d/2})}\\
&+k^c(\tilde{c}_i+\|z_i^k\|_{\mathcal{L}_T^\infty(\dot{B}_{2,1}^{d/2})})\big(\frac{1}{\tilde{\theta}^2}+\|g(\omega^k)\|_{\mathcal{L}_T^\infty(\dot{B}_{2,1}^{d/2})}\big)\|\nabla\omega^k\|_{\mathcal{L}_T^2(\dot{B}_{2,1}^{d/2})}^2
\end{align*}
Then the assumption on the equilibrium state we estimate
\begin{align*}
\|F_i^k\|_{\mathcal{L}_T^1(\dot{B}_{2,1}^{d/2})}&\lesssim k^c \sum_j h^2\|z_j^k\|_{\mathcal{L}_T^1(\dot{B}_{2,1}^{d/2})}+k^\theta h^2 \|\omega^k\|_{\mathcal{L}_T^1(\dot{B}_{2,1}^{d/2})}\\
&+k^c\big(h^2+\|f(\omega^k)\|_{\mathcal{L}_T^\infty(\dot{B}_{2,1}^{d/2})}\big)\|\nabla z_i^k\|_{\mathcal{L}_T^2(\dot{B}_{2,1}^{d/2})}\|\nabla\omega^k\|_{\mathcal{L}_T^2(\dot{B}_{2,1}^{d/2})}\\
&+k^c\|z_i^k\|_{\mathcal{L}_T^\infty(\dot{B}_{2,1}^{d/2})} \big(h^2+\|f(\omega^k)\|_{\mathcal{L}_T^\infty(\dot{B}_{2,1}^{d/2})}\big)\|\Delta \omega^k\|_{\mathcal{L}_T^1(\dot{B}_{2,1}^{d/2})}\\
&+(h^{-2}+\|z_i^k\|_{\mathcal{L}_T^\infty(\dot{B}_{2,1}^{d/2})})\big(h^4+\|g(\omega^k)\|_{\mathcal{L}_T^\infty(\dot{B}_{2,1}^{d/2})}\big)\|\nabla\omega^k\|_{\mathcal{L}_T^2(\dot{B}_{2,1}^{d/2})}^2
\intertext{and using Lemma \ref{Besov Lemma} yields}
\|F_i^k\|_{\mathcal{L}_T^1(\dot{B}_{2,1}^{d/2})}&\lesssim k^c \sum_j h^2\|z_j^k\|_{\mathcal{L}_T^1(\dot{B}_{2,1}^{d/2})}+ h^2 \|\omega^k\|_{\mathcal{L}_T^1(\dot{B}_{2,1}^{d/2})}\\
&+k^c\big(h^2+Q(f,\omega^k)\|\omega^k\|_{\mathcal{L}_T^\infty(\dot{B}_{2,1}^{d/2})}\big)\|\nabla z_i^k\|_{\mathcal{L}_T^2(\dot{B}_{2,1}^{d/2})}\|\nabla\omega^k\|_{\mathcal{L}_T^2(\dot{B}_{2,1}^{d/2})}\\
&+k^c\|z_i^k\|_{\mathcal{L}_T^\infty(\dot{B}_{2,1}^{d/2})} \big(h^2+Q(f,\omega^k)\|\omega^k\|_{\mathcal{L}_T^\infty(\dot{B}_{2,1}^{d/2})}\big)\|\Delta \omega^k\|_{\mathcal{L}_T^1(\dot{B}_{2,1}^{d/2})}\\
&+(h^{-2}+\|z_i^k\|_{\mathcal{L}_T^\infty(\dot{B}_{2,1}^{d/2})})\big(h^4+Q(g,\omega^k)\|\omega^k\|_{\mathcal{L}_T^\infty(\dot{B}_{2,1}^{d/2})}\big)\|\nabla\omega^k\|_{\mathcal{L}_T^2(\dot{B}_{2,1}^{d/2})}^2
\end{align*}
We observe that by the assumption on $z_i^k$ and $\omega^k$ for any fixed $k$ we have
\begin{align*}
\|F_i^k\|_{\mathcal{L}_T^1(\dot{B}_{2,1}^{d/2})}&\lesssim k^c h^2 h^2+k^c\big(h^2+Q(f,\omega^k)h^2\big)h^2 h^2 + h^2 h^2\\
&+k^c h^2 \big(h^2+Q(f,\omega^k)h^2\big)h^2+(h^{-2}+h^2)\big(h^4+Q(g,\omega^k)h^2\big)h^4
\end{align*}
Thus we obtain that
\begin{align}\label{eq conc 2}
\|F_i^k\|_{\mathcal{L}_T^1(\dot{B}_{2,1}^{d/2})}\lesssim h^4
\end{align}
Combining the estimate from equation \eqref{eq conc 1} with the estimate in equation \eqref{eq conc 2} yields
\begin{align}
\|z_i^{k+1}\|_{\mathcal{L}_T^\infty(\dot{B}_{2,1}^{d/2})}+\|\partial_t z_i^{k+1}\|_{\mathcal{L}_T^1(\dot{B}_{2,1}^{d/2})}+k^c\|\Delta z_i^{k+1}\|_{\mathcal{L}_T^1(\dot{B}_{2,1}^{d/2})}\leq C h^4
\end{align}
and thus $\|z_i^{k+1}\|_{\mathcal{B}}\leq h^2$ which concludes the proof of the first estimate in \eqref{est 1}.\\
Now, we consider the difference between two solutions $\delta z_i^{k+1}=z_i^{k+2}-z_i^{k+1}$.
Then $\delta z_i^{k+1}$ is a solution to
\begin{align*}
\partial_t \delta z_i^{k+1} -k^c\Delta \delta z_i^{k+1}= \delta F_i^k
\end{align*}
where 
\begin{align*}
\delta F_i^k&= -\sigma_i\bigg[ k^c\sum_{j}\sigma_j\frac{\delta z^k_j}{\tilde{c}_j}-k^\theta\frac{\delta \omega^k}{\tilde{\theta}}\bigg] -k^c \delta z_i^k \frac{|\nabla \omega^{k+1}|^2}{(\omega^{k+1}+\tilde{\theta})^2}\\
& +k^c\bigg[\nabla \delta z_i^k\cdot \frac{\nabla \omega^{k+1}}{\omega^{k+1}+\tilde{\theta}}-\nabla z_i^k\cdot\bigg(\frac{\nabla \delta\omega^k}{\omega^{k+1}+\tilde{\theta}}+\frac{\nabla \omega^k \delta\omega^k}{(\omega^{k+1}+\tilde{\theta})(\omega^k+\tilde{\theta})}\bigg)\bigg]\\
&+k^c\bigg[\delta z_i^k \frac{\Delta\omega^{k+1}}{\omega^{k+1}+\tilde{\theta}}+z_i^k \bigg(\frac{\Delta\delta\omega^k}{\omega^{k+1}+\tilde{\theta}}+\frac{\Delta\omega^k\delta\omega^k}{(\omega^{k+1}+\tilde{\theta})(\omega^{k}+\tilde{\theta})}\bigg)\bigg]\\
&-k^c(z_i^k+\tilde{c}_i)\bigg(\frac{\nabla \delta \omega^{k+1}\cdot(\nabla \omega^{k+1}+\nabla \omega^k)}{(\omega^{k+1}+\tilde{\theta})^2}-\frac{|\nabla \omega^k|^2(2\tilde{\theta}+\omega^{k+1}+\omega^k)}{(\omega^{k+1}+\tilde{\theta})^2(\omega^{k}+\tilde{\theta})^2}\delta\omega^k\bigg)
\end{align*} 
This can be rewritten as follows
\begin{align*}
\delta F_i^k=& -\sigma_i\bigg[ k^c\sum_{j}\sigma_j\frac{\delta z^k_j}{\tilde{c}_j}-k^\theta\frac{\delta \omega^k}{\tilde{\theta}}\bigg] +k^c\nabla \delta z_i^k\cdot \nabla \omega^{k+1}(\frac{1}{\tilde{\theta}}+f(\omega^{k+1}))\\
&-k^c\nabla z_i^k\cdot\bigg(\nabla \delta\omega^k (\frac{1}{\tilde{\theta}}+f(\omega^{k+1}))+\nabla \omega^k \delta\omega^k (\frac{1}{\tilde{\theta}}+g(\omega^{k+1},\omega^k))\bigg)\\
&+k^c\delta z_i^k \Delta\omega^{k+1}(\frac{1}{\tilde{\theta}}+f(\omega^{k+1}))\\
&+ k^c z_i^k\bigg(\Delta\delta\omega^k (\frac{1}{\tilde{\theta}}+f(\omega^{k+1}))+\Delta\omega^k\delta\omega^k (\frac{1}{\tilde{\theta}^2}+g(\omega^{k+1},\omega^k))\bigg)\\
&-k^c\delta z_i^k |\nabla \omega^{k+1}|^2\big(\frac{1}{\tilde{\theta}^2}+g(\omega^{k+1})\big)\\
&-k^c(z_i^k+\tilde{c}_i)\bigg(\nabla \delta \omega^{k}\cdot(\nabla \omega^{k+1}+\nabla \omega^k)\big(\frac{1}{\tilde{\theta}^2}+g(\omega^{k+1})\big)\\
&-k^c(z_i^k+\tilde{c}_i)|\nabla \omega^k|^2(2\tilde{\theta}+\omega^{k+1}+\omega^k)\big(\frac{1}{\tilde{\theta}^2}+g(\omega^{k+1})\big)\big(\frac{1}{\tilde{\theta}^2}+g(\omega^{k1})\big)   \delta\omega^k
\end{align*}
Again applying Theorem \ref{Besov Theorem} yields
\begin{align}\label{eq conc diff 1}
\|\delta z_i^{k+1}\|_{\mathcal{L}_T^\infty(\dot{B}_{2,1}^{d/2})}+\|\partial_t\delta  z_i^{k+1}\|_{\mathcal{L}_T^1(\dot{B}_{2,1}^{d/2})}+k^c\|\Delta\delta z_i^{k+1}\|_{\mathcal{L}_T^1(\dot{B}_{2,1}^{d/2})}\leq C \|\delta F_i^k\|_{\mathcal{L}_T^1(\dot{B}_{2,1}^{d/2})}
\end{align}
where we can estimate further
\begin{align*}
\|\delta F_i^k\|_{\mathcal{L}_T^1(\dot{B}_{2,1}^{d/2})}\lesssim&  \sum_{j}\sigma_j\frac{\|\delta z^k_j\|_{\mathcal{L}_T^1(\dot{B}_{2,1}^{d/2})}}{\tilde{c}_j}+\frac{\|\delta \omega^k\|_{\mathcal{L}_T^1(\dot{B}_{2,1}^{d/2})}}{\tilde{\theta}}\\
& +\|\nabla \delta z_i^k\|_{\mathcal{L}_T^2(\dot{B}_{2,1}^{d/2})} \|\nabla \omega^{k+1}\|_{\mathcal{L}_T^2(\dot{B}_{2,1}^{d/2})}(\frac{1}{\tilde{\theta}}+\|f(\omega^{k+1})\|_{\mathcal{L}_T^\infty(\dot{B}_{2,1}^{d/2})})\\
&+\|\nabla z_i^k\|_{\mathcal{L}_T^2(\dot{B}_{2,1}^{d/2})}\|\nabla \delta\omega^k \|_{\mathcal{L}_T^2(\dot{B}_{2,1}^{d/2})}(\frac{1}{\tilde{\theta}}+\|f(\omega^{k+1})\|_{\mathcal{L}_T^\infty(\dot{B}_{2,1}^{d/2})})\\
&+\|\nabla z_i^k\|_{\mathcal{L}_T^2(\dot{B}_{2,1}^{d/2})}\|\nabla \omega^k\|_{\mathcal{L}_T^2(\dot{B}_{2,1}^{d/2})} \|\delta\omega^k\|_{\mathcal{L}_T^\infty(\dot{B}_{2,1}^{d/2})}\\
&~~~~\times (\frac{1}{\tilde{\theta}}+\|g(\omega^{k+1},\omega^k)\|_{\mathcal{L}_T^\infty(\dot{B}_{2,1}^{d/2})})\\
&+ \|\delta z_i^k\|_{\mathcal{L}_T^\infty(\dot{B}_{2,1}^{d/2})}\| \Delta\omega^{k+1}\|_{\mathcal{L}_T^1(\dot{B}_{2,1}^{d/2})}(\frac{1}{\tilde{\theta}}+\|f(\omega^{k+1})\|_{\mathcal{L}_T^\infty(\dot{B}_{2,1}^{d/2})})\\
&+\|z_i^k\|_{\mathcal{L}_T^\infty(\dot{B}_{2,1}^{d/2})}\|\Delta\delta\omega^k\|_{\mathcal{L}_T^1(\dot{B}_{2,1}^{d/2})} (\frac{1}{\tilde{\theta}}+\|f(\omega^{k+1})\|_{\mathcal{L}_T^\infty(\dot{B}_{2,1}^{d/2})})\\
&+(\|z_i^k\|_{\mathcal{L}_T^\infty(\dot{B}_{2,1}^{d/2})}+\tilde{c}_i)\|\Delta\omega^k\|_{\mathcal{L}_T^1(\dot{B}_{2,1}^{d/2})}\|\delta\omega^k \|_{\mathcal{L}_T^\infty(\dot{B}_{2,1}^{d/2})}\\
&~~~~\times(\frac{1}{\tilde{\theta}^2}+\|g(\omega^{k+1},\omega^k)\|_{\mathcal{L}_T^\infty(\dot{B}_{2,1}^{d/2})})\\
&+\|\delta z_i^k\|_{\mathcal{L}_T^\infty(\dot{B}_{2,1}^{d/2})} \|\nabla \omega^{k+1}\|_{\mathcal{L}_T^2(\dot{B}_{2,1}^{d/2})}^2\big(\frac{1}{\tilde{\theta}^2}+\|g(\omega^{k+1})\|_{\mathcal{L}_T^\infty(\dot{B}_{2,1}^{d/2})}\big)\\
&+(\|z_i^k\|_{\mathcal{L}_T^\infty(\dot{B}_{2,1}^{d/2})}+\tilde{c}_i)(\|\nabla \omega^{k+1}\|_{\mathcal{L}_T^2(\dot{B}_{2,1}^{d/2})}+\|\nabla \omega^k\|_{\mathcal{L}_T^2(\dot{B}_{2,1}^{d/2})})\\
&~~~~\times\|\nabla \delta \omega^{k}\|_{\mathcal{L}_T^2(\dot{B}_{2,1}^{d/2})}\big(\frac{1}{\tilde{\theta}^2}+\|g(\omega^{k+1})\|_{\mathcal{L}_T^\infty(\dot{B}_{2,1}^{d/2})}\big)\bigg)\\
&+(\|z_i^k\|_{\mathcal{L}_T^\infty(\dot{B}_{2,1}^{d/2})}+\tilde{c}_i(2\tilde{\theta}+\|\omega^{k+1}\|_{\mathcal{L}_T^\infty(\dot{B}_{2,1}^{d/2})}+\|\omega^k\|_{\mathcal{L}_T^\infty(\dot{B}_{2,1}^{d/2})})\\
&~~~~\times \big(\frac{1}{\tilde{\theta}^2}+\|g(\omega^{k+1})\|_{\mathcal{L}_T^\infty(\dot{B}_{2,1}^{d/2})}\big))\|\nabla \omega^k\|_{\mathcal{L}_T^2(\dot{B}_{2,1}^{d/2})}^2\\
&~~~~\times \big(\frac{1}{\tilde{\theta}^2}+\|g(\omega^{k+1})\|_{\mathcal{L}_T^\infty(\dot{B}_{2,1}^{d/2})}\big)   \|\delta\omega^k\|_{\mathcal{L}_T^\infty(\dot{B}_{2,1}^{d/2})}
\end{align*}
By using the assumptions on the equilibrium state and by applying the previous estimates we obtain
\begin{align*}
\|\delta F_i^k\|_{\mathcal{L}_T^1(\dot{B}_{2,1}^{d/2})}&\lesssim (h^2+h^5+h^7)\sum_{j}\|\delta z^k_j\|_{\mathcal{L}_T^1(\dot{B}_{2,1}^{d/2})}+h^5 \|\nabla \delta z_i^k\|_{\mathcal{L}_T^2(\dot{B}_{2,1}^{d/2})}\\
&~~+\|\nabla \delta\omega^k \|_{\mathcal{L}_T^2(\dot{B}_{2,1}^{d/2})}(h^3+h^4+h^5+h^7)+\|\Delta\delta\omega^k\|_{\mathcal{L}_T^1(\dot{B}_{2,1}^{d/2})}h^4\\
&~~+(h^2+h^3+h^5+h^6)\|\delta \omega^k\|_{\mathcal{L}_T^1(\dot{B}_{2,1}^{d/2})}
\end{align*}
Now, taking into account the induction assumption yields the following
\begin{align}
\|\delta F_i^k\|_{\mathcal{L}_T^1(\dot{B}_{2,1}^{d/2})}\lesssim h^{k+2}
\end{align}
Combining the above estimates yields
\begin{align}
\|\delta z_i^{k+1}\|_{\mathcal{L}_T^\infty(\dot{B}_{2,1}^{d/2})}+\|\partial_t\delta  z_i^{k+1}\|_{\mathcal{L}_T^1(\dot{B}_{2,1}^{d/2})}+k^c\|\Delta\delta z_i^{k+1}\|_{\mathcal{L}_T^1(\dot{B}_{2,1}^{d/2})}\leq h^{k+1}
\end{align}
which concludes the proof of the induction. \\

\textit{Temperature equation:} We proceed in a similar fashion as for the concentration equation.
Let $\omega^k$ be the approximate solution to the previous step. 
Then by Theorem \ref{Besov Theorem} we have that the solution to the next step $\omega^{k+1}$ in the approximate temperature equation exists and that the norm of $\|\omega^{k+1}\|_\mathcal{B}$ is bounded by
\begin{align*}
\|\omega^{k+1}\|_{\mathcal{L}_T^\infty(\dot{B}_{2,1}^{d/2})} &+\|\partial_t\omega^{k+1}\|_{\mathcal{L}_T^1(\dot{B}_{2,1}^{d/2})}+\|\Delta\omega^{k+1}\|_{\mathcal{L}_T^1(\dot{B}_{2,1}^{d/2})}\\
&\leq C\big[\|\omega_0\|_{\dot{B}_{2,1}^{d/2}}+\|G^{k}\|_{\mathcal{L}_T^1(\dot{B}_{2,1}^{d/2})}\big].
\end{align*}
By the assumption on the initial perturbation in the temperature we obtain
\begin{align}\begin{split}
\|\omega^{k+1}\|_{\mathcal{L}_T^\infty(\dot{B}_{2,1}^{d/2})} &+\|\partial_t\omega^{k+1}\|_{\mathcal{L}_T^1(\dot{B}_{2,1}^{d/2})}+\|\Delta\omega^{k+1}\|_{\mathcal{L}_T^1(\dot{B}_{2,1}^{d/2})}\\
&\leq C\big[h^4+\|G^{k}\|_{\mathcal{L}_T^1(\dot{B}_{2,1}^{d/2})}\big].\end{split}
\end{align}
Next, we claim that the forcing term can be bounded as follows
\begin{align*}
\|G^k\|_{\mathcal{L}_T^1(\dot{B}_{2,1}^{d/2})}\leq&  \sum_i \|\nabla z_i^k\|_{\mathcal{L}_T^2(\dot{B}_{2,1}^{d/2})}\|\nabla \omega^k\|_{\mathcal{L}_T^2(\dot{B}_{2,1}^{d/2})} \bigg(\tilde{c}^{-1}+\|f(c^k)\|_{\mathcal{L}_T^\infty(\dot{B}_{2,1}^{d/2})}\bigg)\\
&+\|\nabla \omega^k\|_{\mathcal{L}_T^2(\dot{B}_{2,1}^{d/2})}^2\bigg(\tilde{\theta}^{-1}+\|f(\omega^k)\|_{\mathcal{L}_T^\infty(\dot{B}_{2,1}^{d/2})}\bigg)\\
&+ \|f(c^k)\|_{\mathcal{L}_T^\infty(\dot{B}_{2,1}^{d/2})}\|\Delta\omega^k\|_{\mathcal{L}_T^1(\dot{B}_{2,1}^{d/2})}\\
& +\bigg(\frac{1}{\tilde{c}}+\|f(c^k)\|_{\mathcal{L}_T^\infty(\dot{B}_{2,1}^{d/2})}\bigg) (\|\omega^k\|_{\mathcal{L}_T^\infty(\dot{B}_{2,1}^{d/2})}+\tilde{\theta})\\
&~~~\times\bigg(\sum_{j}\frac{\|z_j^k\|_{\mathcal{L}_T^1(\dot{B}_{2,1}^{d/2})}}{\tilde{c}_j}+\frac{\|\omega^k\|_{\mathcal{L}_T^1(\dot{B}_{2,1}^{d/2})}}{\tilde{\theta}}\bigg) \\
& + \bigg(\frac{1}{\tilde{c}}+\|f(c^k)\|_{\mathcal{L}_T^\infty(\dot{B}_{2,1}^{d/2})}\bigg)\sum_i \|\nabla z_i^k\|_{\mathcal{L}_T^2(\dot{B}_{2,1}^{d/2})}\|\nabla\omega^k\|_{\mathcal{L}_T^2(\dot{B}_{2,1}^{d/2})}\\
&+ \bigg(\frac{1}{\tilde{c}}+\|f(c^k)\|_{\mathcal{L}_T^\infty(\dot{B}_{2,1}^{d/2})}\bigg)\sum_i\|\omega^k\|_{\mathcal{L}_T^\infty(\dot{B}_{2,1}^{d/2})} \|\Delta z_i^k\|_{\mathcal{L}_T^1(\dot{B}_{2,1}^{d/2})},
\end{align*}
where we assume that $\eta_i=1$ and thus the additional term can be dropped.
The assumptions on the equilibrium state and applying Lemma \ref{Besov Lemma} then yields
\begin{align*}
\|G^k\|_{\mathcal{L}_T^1(\dot{B}_{2,1}^{d/2})}\lesssim & \sum_i \|\nabla z_i^k\|_{\mathcal{L}_T^2(\dot{B}_{2,1}^{d/2})}\|\nabla \omega^k\|_{\mathcal{L}_T^2(\dot{B}_{2,1}^{d/2})} \bigg(h^2+Q(f,c^k)\|c^k\|_{\mathcal{L}_T^\infty(\dot{B}_{2,1}^{d/2})}\bigg)\\
&+\|\nabla \omega^k\|_{\mathcal{L}_T^2(\dot{B}_{2,1}^{d/2})}^2\bigg(h^2+Q(f,\omega^k)\|\omega^k\|_{\mathcal{L}_T^\infty(\dot{B}_{2,1}^{d/2})}\bigg)\\
&+ Q(f,c^k)\|c^k\|_{\mathcal{L}_T^\infty(\dot{B}_{2,1}^{d/2})}\|\Delta\omega^k\|_{\mathcal{L}_T^1(\dot{B}_{2,1}^{d/2})}\\
&+\bigg(h^2+Q(f,c^k)\|c^k\|_{\mathcal{L}_T^\infty(\dot{B}_{2,1}^{d/2})}\bigg) (\|\omega^k\|_{\mathcal{L}_T^\infty(\dot{B}_{2,1}^{d/2})}+h^{-2})\\
&~~~\times h^2\bigg(\sum_{j}\|z_j^k\|_{\mathcal{L}_T^1(\dot{B}_{2,1}^{d/2})}+\|\omega^k\|_{\mathcal{L}_T^1(\dot{B}_{2,1}^{d/2})}\bigg) \\
& + \bigg(h^2+Q(f,c^k)\|c^k\|_{\mathcal{L}_T^\infty(\dot{B}_{2,1}^{d/2})}\bigg)\sum_i \|\nabla z_i^k\|_{\mathcal{L}_T^2(\dot{B}_{2,1}^{d/2})}\|\nabla\omega^k\|_{\mathcal{L}_T^2(\dot{B}_{2,1}^{d/2})}\\
&+ \bigg(h^2+Q(f,c^k)\|c^k\|_{\mathcal{L}_T^\infty(\dot{B}_{2,1}^{d/2})}\bigg)\sum_i\|\omega^k\|_{\mathcal{L}_T^\infty(\dot{B}_{2,1}^{d/2})} \|\Delta z_i^k\|_{\mathcal{L}_T^1(\dot{B}_{2,1}^{d/2})}
\end{align*}
Using the control of $z_i^k$ and $\omega^k$ we have
\begin{align}
\|G^k\|_{\mathcal{L}_T^1(\dot{B}_{2,1}^{d/2})}&\leq h^4
\end{align}
Hence we obtain
\begin{align}
\|\omega^{k+1}\|_{\mathcal{L}_T^\infty(\dot{B}_{2,1}^{d/2})} +\|\partial_t\omega^{k+1}\|_{\mathcal{L}_T^1(\dot{B}_{2,1}^{d/2})}+\|\Delta\omega^{k+1}\|_{\mathcal{L}_T^1(\dot{B}_{2,1}^{d/2})}\leq C h^4.
\end{align}
and thus
\begin{align}
\|\omega^{k+1}\|_{\mathcal{B}}\leq h^2
\end{align}
which completes the proof of the second estimate in \eqref{est 1}.\\
Finally, we consider the difference between two approximate solutions and set $\delta\omega^{k+1}=\omega^{k+2}-\omega^{k+1}$.
Then $\delta\omega^{k+1}$ is a solution to
\begin{align*}
\partial_t \delta\omega^{k+1}-\tilde{\kappa}\Delta\delta\omega^{k+1}=\delta G^k,
\end{align*}
where
\begin{align*}
\delta G^k=& k^\theta\big(\tilde{c}^{-1}+f(c^{k+1})\big)\sum_i \bigg( \nabla\delta z_i^k\cdot\nabla \omega^{k+1} +\nabla z_i^k\cdot\nabla \delta \omega^k\bigg)\\
&+ k^\theta \sum_i\nabla z_i^k\cdot\nabla \omega^k \delta\omega^k\big(\tilde{c}^{-2}+g(c^{k+1},c^k)\big)\\
&+k^\theta \nabla \delta\omega^k\cdot\big(\nabla \omega^{k+1}+\nabla\omega^k\big)\big(\tilde{\theta}^{-1}+f(\omega^{k+1})\big)\\
& +|\nabla \omega^k|^2\delta\omega^k\big(\tilde{\theta}^{-1}+g(\omega^{k+1},\omega^k)\big)\\
&+\kappa f(c^{k+1})\Delta\delta\omega^k+\kappa\Delta\omega^k \sum_i\delta z_i^k g(c^{k+1},c^k)\\
&+(\omega^{k+1}+\tilde{\theta})\big(\tilde{c}^{-1}+f(c^{k+1})\big)\bigg(\sum_j \sigma_j\frac{\delta z_j^k}{\tilde{c}_j}-\frac{\delta \omega^k}{\tilde{\theta}}\bigg)\\
&+\bigg(\sum_j\sigma_j\frac{z_j^k}{\tilde{c}_j}-\frac{\omega^k}{\tilde{\theta}}\bigg)\bigg(\delta\omega^k\big(\tilde{c}^{-1}+f(c^{k+1})\big)\\
&~~~ +(\omega^k+\tilde{\theta})\sum_i\delta z_i^k \big(\tilde{c}^{-2}+g(c^{k+1},c^k)\big)\bigg)\\
&+\sum_i\big(\omega^{k+1}\Delta\delta z_i^k +\delta\omega^k\Delta z_i^k\big)\big(\tilde{c}^{-1}+f(c^{k+1})\big)\\
&+\omega^k \sum_i \Delta z_i^k\sum \delta z_i^k \big(\tilde{c}^{-2}+g(c^{k+1},c^k)\big)
\end{align*}
Then by applying Theorem \ref{Besov Theorem} we have the following estimate
\begin{align}
\|\delta\omega^{k+1}\|_{\mathcal{L}_T^\infty(\dot{B}_{2,1}^{d/2})}+ \|\partial_t \delta\omega^{k+1}\|_{\mathcal{L}_T^1(\dot{B}_{2,1}^{d/2})}+\|\tilde{\kappa}\Delta\delta\omega^{k+1}\|_{\mathcal{L}_T^1(\dot{B}_{2,1}^{d/2})}\leq C\|\delta G^k\|_{\mathcal{L}_T^1(\dot{B}_{2,1}^{d/2})},
\end{align} 
where we estimate the last term as follows
\begin{align*}
\|\delta G^k\|_{\mathcal{L}_T^1(\dot{B}_{2,1}^{d/2})}&\lesssim  \big(\tilde{c}^{-1}+\|f(c^{k+1})\|_{\mathcal{L}_T^\infty(\dot{B}_{2,1}^{d/2})}\big)\sum_i  \|\nabla\delta z_i^k\|_{\mathcal{L}_T^2(\dot{B}_{2,1}^{d/2})}\|\nabla \omega^{k+1}\|_{\mathcal{L}_T^2(\dot{B}_{2,1}^{d/2})} \\
&+\big(\tilde{c}^{-1}+\|f(c^{k+1})\|_{\mathcal{L}_T^1(\dot{B}_{2,1}^{d/2})}\big)\sum_i\|\nabla z_i^k\|_{\mathcal{L}_T^2(\dot{B}_{2,1}^{d/2})}\|\nabla \delta \omega^k\|_{\mathcal{L}_T^2(\dot{B}_{2,1}^{d/2})}\\
&+\big(\tilde{c}^{-2}+\|g(c^{k+1},c^k)\|_{\mathcal{L}_T^\infty(\dot{B}_{2,1}^{d/2})}\big) \sum_i\|\nabla z_i^k\|_{\mathcal{L}_T^2(\dot{B}_{2,1}^{d/2})}\\
&~~~\times\|\nabla \omega^k\|_{\mathcal{L}_T^2(\dot{B}_{2,1}^{d/2})}\| \delta\omega^k\|_{\mathcal{L}_T^\infty(\dot{B}_{2,1}^{d/2})}\\
&+\big(\tilde{\theta}^{-1}+\|f(\omega^{k+1})\|_{\mathcal{L}_T^\infty(\dot{B}_{2,1}^{d/2})}\big)\|\nabla \delta\omega^k\|_{\mathcal{L}_T^2(\dot{B}_{2,1}^{d/2})}\\
&~~~\times \big(\|\nabla \omega^{k+1}\|_{\mathcal{L}_T^2(\dot{B}_{2,1}^{d/2})}+\|\nabla\omega^k\|_{\mathcal{L}_T^2(\dot{B}_{2,1}^{d/2})}\big) \\
&+\big(\tilde{\theta}^{-1}+\|g(\omega^{k+1},\omega^k)\|_{\mathcal{L}_T^\infty(\dot{B}_{2,1}^{d/2})}\big)\|\nabla \omega^k\|_{\mathcal{L}_T^2(\dot{B}_{2,1}^{d/2})}^2\|\delta\omega^k\|_{\mathcal{L}_T^\infty(\dot{B}_{2,1}^{d/2})}\\
&+\kappa \|f(c^{k+1})\|_{\mathcal{L}_T^\infty(\dot{B}_{2,1}^{d/2})}\|\Delta\delta\omega^k\|_{\mathcal{L}_T^1(\dot{B}_{2,1}^{d/2})}\\
&+\kappa\|\Delta\omega^k\|_{\mathcal{L}_T^\infty(\dot{B}_{2,1}^{d/2})} \sum_i\|\delta z_i^k\|_{\mathcal{L}_T^\infty(\dot{B}_{2,1}^{d/2})}\| g(c^{k+1},c^k)\|_{\mathcal{L}_T^\infty(\dot{B}_{2,1}^{d/2})}\\
&+(\|\omega^{k+1}\|_{\mathcal{L}_T^\infty(\dot{B}_{2,1}^{d/2})}+\tilde{\theta})\big(\tilde{c}^{-1}+\|f(c^{k+1})\|_{\mathcal{L}_T^\infty(\dot{B}_{2,1}^{d/2})}\big)\\
&~~~\times\bigg(\sum_j \frac{\|\delta z_j^k\|_{\mathcal{L}_T^1(\dot{B}_{2,1}^{d/2})}}{\tilde{c}_j}+\frac{\|\delta \omega^k\|_{\mathcal{L}_T^1(\dot{B}_{2,1}^{d/2})}}{\tilde{\theta}}\bigg)\\
&+\bigg(\sum_j\frac{\|z_j^k\|_{\mathcal{L}_T^\infty(\dot{B}_{2,1}^{d/2})}}{\tilde{c}_j}+\frac{\|\omega^k\|_{\mathcal{L}_T^\infty(\dot{B}_{2,1}^{d/2})}}{\tilde{\theta}}\bigg)\\
&~~~\times\bigg[\|\delta\omega^k\|_{\mathcal{L}_T^1(\dot{B}_{2,1}^{d/2})}\big(\tilde{c}^{-1}+\|f(c^{k+1})\|_{\mathcal{L}_T^\infty(\dot{B}_{2,1}^{d/2})}\big) +(\|\omega^k\|_{\mathcal{L}_T^\infty(\dot{B}_{2,1}^{d/2})}+\tilde{\theta})\\
&~~~~~~\times\sum_i\|\delta z_i^k\|_{\mathcal{L}_T^1(\dot{B}_{2,1}^{d/2})} \big(\tilde{c}^{-2}+\|g(c^{k+1},c^k)\|_{\mathcal{L}_T^\infty(\dot{B}_{2,1}^{d/2})}\big)\bigg]\\
&+\big(\tilde{c}^{-1}+\|f(c^{k+1})\|_{\mathcal{L}_T^\infty(\dot{B}_{2,1}^{d/2})}\big)\sum_i \|\omega^{k+1}\|_{\mathcal{L}_T^\infty(\dot{B}_{2,1}^{d/2})}\|\Delta\delta z_i^k\|_{\mathcal{L}_T^1(\dot{B}_{2,1}^{d/2})} \\
&+\big(\tilde{c}^{-1}+\|f(c^{k+1})\|_{\mathcal{L}_T^\infty(\dot{B}_{2,1}^{d/2})}\big)\sum_i\|\delta\omega^k\|_{\mathcal{L}_T^\infty(\dot{B}_{2,1}^{d/2})}\|\Delta z_i^k\|_{\mathcal{L}_T^1(\dot{B}_{2,1}^{d/2})}\\
&+\big(\tilde{c}^{-2}+\|g(c^{k+1},c^k)\|_{\mathcal{L}_T^\infty(\dot{B}_{2,1}^{d/2})}\big)\|\omega^k \|_{\mathcal{L}_T^\infty(\dot{B}_{2,1}^{d/2})}\\
&~~~\times\sum_i \|\Delta z_i^k\|_{\mathcal{L}_T^1(\dot{B}_{2,1}^{d/2})}\sum_i \|\delta z_i^k\|_{\mathcal{L}_T^\infty(\dot{B}_{2,1}^{d/2})} 
\end{align*}
Using the assumptions on the equilibrium state and the previous estimates yields
\begin{align*}
\|\delta G^k\|_{\mathcal{L}_T^1(\dot{B}_{2,1}^{d/2})}&\lesssim (h^2+h^4)\sum_i \bigg( \|\Delta\delta z_i^k\|_{\mathcal{L}_T^1(\dot{B}_{2,1}^{d/2})}+\|\nabla\delta z_i^k\|_{\mathcal{L}_T^2(\dot{B}_{2,1}^{d/2})}+\|\delta z_i^k\|_{\mathcal{L}_T^\infty(\dot{B}_{2,1}^{d/2})}\bigg)\\
&+(h^2+h^4)\bigg(\|\Delta\delta\omega^k\|_{\mathcal{L}_T^1(\dot{B}_{2,1}^{d/2})}+\|\nabla \delta \omega^k\|_{\mathcal{L}_T^2(\dot{B}_{2,1}^{d/2})}+\|\delta\omega^k\|_{\mathcal{L}_T^\infty(\dot{B}_{2,1}^{d/2})} \bigg)
\end{align*}
By combining this inequality with the induction assumption we obtain
\begin{align}
\|\delta G^k\|_{\mathcal{L}_T^1(\dot{B}_{2,1}^{d/2})}&\lesssim h^{k+2}
\end{align}
and therefore this results in the final estimate
\begin{align}
\|\delta\omega^{k+1}\|_{\mathcal{L}_T^\infty(\dot{B}_{2,1}^{d/2})}+ \|\partial_t \delta\omega^{k+1}\|_{\mathcal{L}_T^1(\dot{B}_{2,1}^{d/2})}+\|\tilde{\kappa}\Delta\delta\omega^{k+1}\|_{\mathcal{L}_T^1(\dot{B}_{2,1}^{d/2})}\leq h^{k+1}
\end{align}
which concludes the proof of the proposition.\\

The next step in the proof of Theorem \ref{main Theorem} is to pass to the limit.
From the uniform estimates obtained in Proposition \ref{existence prop} we can take the limit as $k$ goes to $\infty$.
Since $\big(z_A^k,z_B^k,z_C^k,\omega^k\big)_{k\in \mathbb{N}}$ is a Cauchy sequence the following convergence result holds:
\begin{align}
&z_i^k\to z_i~~\text{in } \mathcal{L}_T^\infty\big(\dot{B}_{2,1}^{d/2}\big),~~(\partial_t z_i^k,\Delta z_i^k)\to (\partial_t z_i,\Delta z_i)~~\in  \mathcal{L}_T^1\big(\dot{B}_{2,1}^{d/2}\big)~~\text{for } i=A,B,C\\
&\omega^k\to \omega~~\text{in } \mathcal{L}_T^\infty\big(\dot{B}_{2,1}^{d/2}\big),~~(\partial_t \omega^k,\Delta \omega^k)\to (\partial_t \omega,\Delta \omega)~~\in  \mathcal{L}_T^1\big(\dot{B}_{2,1}^{d/2}\big)
\end{align}
Therefore, by passing to the limit as $k\to \infty$ we obtain that 
\begin{align*}
\big(z_A,z_B,z_C,\omega\big)=\big(c_A-\tilde{c}_A,c_B-\tilde{c}_B,c_C-\tilde{c}_C,\theta-\tilde{\theta}\big)
\end{align*}
is a classical solution to the reaction-diffusion system with temperature close to equilibrium \eqref{R-D system 1}-\eqref{R-D system 2}\\

The final step in the proof of Theorem \ref{main Theorem} is to show the uniqueness of solutions.
\begin{prop}\label{unique prop}
Let the initial data $\big(z_{A,0},z_{B,0},z_{C,0},\omega_0\big)$ satisfy the assumptions of Theorem \ref{main Theorem} and let $\big(z_A^j,z_B^j,z_C^j,\omega^j\big)$ for $j=1,2$ be two classical solutions to the same initial data belonging to the space $\mathcal{B}$ defined in \eqref{B norm def}.
Setting $\delta z_i=z_i^1-z_i^2$ for $i=A,B,C$ and $\delta \omega=\omega^1-\omega^2$ the difference between the two solutions it follows that
\begin{align}
\sum_i\|\delta z_i\|_{\mathcal{B}}+\|\delta\omega\|_{\mathcal{B}}\lesssim h^2\bigg( \sum_i\|\delta z_i\|_{\mathcal{B}}+\|\delta\omega\|_{\mathcal{B}}\bigg).
\end{align} 
\end{prop}
This implies that for $h>0$ small enough we have
\begin{align*}
\sum_i\|\delta z_i\|_{\mathcal{B}}+\|\delta\omega\|_{\mathcal{B}}=0
\end{align*}
and therefore the two solutions coincide.\\
The proof of the proposition follows by repeating the arguments used to bound the differences of two approximate solutions in equations \eqref{R-D system 1} and \eqref{R-D system 2}.\\
This concludes the proof of the well-posedness result for the chemical reaction-diffusion system with temperature.

\subsection{Conclusion and Remarks}
From the general model of the non-isothermal reaction-diffusion system we can obtain the ideal gas model by considering only one species with density $\rho$ and by setting the reaction rate to zero, see \cite{Sulzbach} for more details in the derivation.
Thus the system has the following form
\begin{align*}
\partial_t \rho-k^\rho	\Delta \rho &= k^\rho \nabla\cdot\big(\rho \nabla \ln \theta)\\
k^\theta\rho\bigg(\partial_t\theta-k^\theta \frac{\nabla \rho\cdot\nabla \theta}{\rho}-k^\theta \frac{|\nabla \theta|^2}{\theta}\bigg)&=\kappa \Delta \theta +(k^\rho)^2(\eta -1)\frac{\nabla(\rho\theta)}{\rho\theta}+(k^\rho)^2\Delta(\rho\theta).
\end{align*}
Similar, by using a different constitutive relation in the dissipation we can obtain the ideal gas system discussed in \cite{Lai}
\begin{align*}
\partial_t \rho &=\Delta(\rho \theta)\\
k^\theta\partial_t(\rho\theta)-k^\rho(k^\rho+k^\theta)\nabla\cdot\bigg(\theta\nabla(\rho	\theta)\bigg)&=\nabla\cdot (\kappa\nabla \theta).
\end{align*}
We observe that the well-posedness result for the reaction-diffusion systems (Theorem \ref{main Theorem}) can be applied to both systems, yielding the existence of solutions to a system with small perturbations.

For a different approach to these systems we refer to \cite{Sulzbach}, where the existence of weak solutions to the Brinkman-Fourier system on a bounded domain is proven by using energy estimates rather then scaling arguments.

As for future work we want to extend the derivation of non-isothermal fluid mechanics to non-local systems with the porous media equation and the Poisson-Nerst-Plank equation as examples, see \cite{Deng} for the case without temperature.

\section*{Acknowledgments}
The authors would like to express their thank to Dr Yiwei Wang and Professor Tengfei Zhang for the inspiring discussions and new insights.
The work was partially supported by DMS-1950868, and the United States – Israel Binational Science Foundation (BSF) {\#}2024246 .

\printbibliography

\end{document}